%% file: SLOD_NLSE_arXiv.tex
\documentclass[a4paper,11pt]{article}
\usepackage{amsmath}
\usepackage{amsfonts}
\usepackage{titlesec}
\usepackage{amssymb}
\usepackage{amsthm}
\usepackage{amsmath}
\usepackage{bbm} 

\usepackage{algorithm}
\usepackage[verbose,a4paper,tmargin=35mm,bmargin=25mm,lmargin=28mm,rmargin=28mm]{geometry}

\usepackage{tikz,pgfplots}
\usepackage{color}

\newcommand{\D}{\mathcal{D}}

\newcommand{\R}{\mathbb{R}}

\usepackage[utf8]{inputenc}

\usepackage{xfrac}
\usepackage{nicefrac}
\usepackage{mathtools}

\usepackage{array}
\makeatletter
\newcommand{\thickhline}{%
	\noalign {\ifnum 0=`}\fi \hrule height 1pt
	\futurelet \reserved@a \@xhline
}
\newcolumntype{"}{@{\hskip\tabcolsep\vrule width 1pt\hskip\tabcolsep}}
\makeatother

\newcommand{\ci}{\mathrm{i}} 

\usepackage{framed}
\usepackage{graphicx, caption, subcaption}
\usepackage{hyperref}
\usepackage{float}

\newtheorem{theorem}{Theorem}[section]
\newtheorem{corollary}[theorem]{Corollary}
\newtheorem{lemma}[theorem]{Lemma}
\newtheorem{remark}[theorem]{Remark}

\newcommand{\SLOD}{{\OD,\ell}}

\newcommand{\VOD}{\mathcal{V}_{\OD}}
\newcommand{\VSLOD}{\mathcal{V}_{\OD,\ell}}
\newcommand{\OD}{{\text{\tiny OD}}}
\newcommand{\N}{\mathbb{N}}

\newcommand{\dx}{\hspace{2pt}\mbox{d}x}

\newcommand{\Ccomment}[1]{\textcolor{magenta}{Ch: #1}}

\newcommand{\Dcomment}[1]{\textcolor{red}{D: #1}}

\usepackage{diagbox}

\usepackage{epstopdf}
 
\usepackage{authblk}
\title{Super-localised wave function approximation of Bose-Einstein condensates\footnote{
		The work of D. Peterseim and C. Zimmer is part of a project that has received funding from the European Research Council (ERC) under the European Union's Horizon 2020 research and innovation programme (Grant agreement No.~865751 -- RandomMultiScales). J. W\"arneg{\aa}rd acknowledges support by the Knut and Alice Wallenberg Foundation.}}

\author[$\dagger$$\star$]{D. Peterseim}
\author[$\ddag$]{J. W\"arneg{\aa}rd}
\author[$\dagger$]{C. Zimmer}

\affil[$\dagger$]{{\small Institute of Mathematics, University of Augsburg, Universit\"atsstr.~12a, 86159 Augsburg, Germany}}
\affil[$\star$]{ {\small Centre for Advanced Analytics and Predictive Sciences (CAAPS), University of Augsburg, Universit\"atsstr.~12a, 86159 Augsburg, Germany}}
\affil[$\ddag$]{{\small Department of Applied Physics and Applied Mathematics, Columbia University, New York, NY 10027, USA}}

\begin{document}

\maketitle 

\begin{abstract}
This paper presents a novel spatial discretisation method for the reliable and efficient simulation of Bose-Einstein condensates modelled by the Gross-Pitaevskii equation and the corresponding nonlinear eigenvector problem. The method combines the high-accuracy properties of numerical homogenisation methods with a novel super-localisation approach for the calculation of the basis functions.  A rigorous numerical analysis demonstrates superconvergence of the approach compared to classical polynomial and multiscale finite element methods, even in low regularity regimes. Numerical tests reveal the method's competitiveness with spectral methods, particularly in capturing critical physical effects in extreme conditions, such as vortex lattice formation in fast-rotating potential traps. The method's potential is further highlighted through a dynamic simulation of a phase transition from Mott insulator to Bose-Einstein condensate, emphasising its capability for reliable exploration of physical phenomena.
\end{abstract}

\section{Introduction}
In this paper we consider the calculation of the ground states and dynamics of Bose-Einstein condensates (BECs). Mathematically, this corresponds to first computing the wave function of the condensate as the solution of a constrained nonlinear minimisation problem representing either ground or excited states. The dynamical evolution of these eigenstates under external perturbation or excitation of the system is then modelled by a nonlinear time-dependent PDE. Among the many applications, we note in particular the keen interest of the physics community in ultracold atoms forming BECs and, more recently, the study of quasi-particles in semiconductor physics, such as exciton-polaritons, which behave as BECs under certain conditions~\cite{PolaritonBEC}. It is noted that the numerical treatment of the latter example may require methods that are robust to low regularity of the wave function, e.g.~due to the presence of Gaussian noise~\cite{NoiseBEC}. A distinctive feature of the novel methodology of this work is its higher order convergence results under very low regularity assumptions and at low computational cost.
Numerical methods for calculating ground states and dynamics of BECs have attracted much interest in both the physics and computational mathematics communities. It is therefore difficult to give an exhaustive list of previous work. Categorised by spatial discretisation, the most popular approach has arguably been that of Fourier-based or spectral methods \cite{BaC12, Spectral2, gpelab,GPELabII}, but finite difference methods (FDM) \cite{FDM2, FDM3, BaC12,BaC13, FDMGPE}, adaptive finite element methods (FEM) \cite{DANAILA20106946} and even combinations of Fourier methods and FEM \cite{BaoDuYanzhi} have also been investigated. Overviews and meaningful comparisons of these approaches have been made in \cite{BaoNum, ANTOINE20132621,BaD04}. In general, these considerations have found that spectral methods are the most efficient for simple geometries and smooth potentials, although certain situations such as a rapidly varying optical lattice potential have been found to favour an FDM over a spectral approach \cite{FDM3}.
Recently, advanced problem-adapted FEMs, which are both fast and, unlike Fourier-based methods, perform well in cases of low regularity, have been proposed to solve both the time-dependent cubic nonlinear Schrödinger equation and the associated Gross-Pitaevskii nonlinear eigenvector problem \cite{HeMaPe14,Superconv,HeWaDo22,HeP21}. The advantages of these problem-adapted finite element methods over classical polynomial-based approaches derive mainly from two ideas. First, the full flexibility of finite elements is exploited to adapt their shape functions to a partial differential operator (e.g.~a suitable linearisation of the full nonlinear operator) associated with the particular problem. This adaptation, in the spirit of numerical homogenisation by localised orthogonal decomposition \cite{LocalElliptic,MalP20,ActaNumericaLOD}, preserves the favourable optimal convergence rates known for smooth solutions to problems with low regularity due to the presence of rough potentials. Second, instances of the wave function's  density are replaced by their $L^2$ projection onto this adapted finite element space. This quadrature-like simplification dramatically speeds up the assembly of nonlinear terms, which are often the complexity bottleneck in practice.

The present work not only provides a mathematically rigorous foundation for this second step, but also substantially advances this demonstrably promising approach. On the theoretical side, we prove that the optimal convergence rates can be recovered without increasing the computational complexity when computing the minimum energy in a low regularity regime. In particular, our proof of the convergence of $\mathcal{O}(H^6)$ holds for potentials in $L^{2+\sigma}(\D)$, where $\sigma$ is non-negative in $1$d, $2$d and positive in $3$d. Here $H$ denotes the mesh size, which can be coarse in the sense that it need not resolve features of the solution such as kinks and fast variations. Previously, optimal convergence rates of $\mathcal{O}(H^6)$ were known only for the much more expensive exact evaluation of the nonlinearity in finite element space~\cite{HeP21}.

On the practical side, it is shown that the construction of the problem-adapted FE space can be significantly improved by using the recent super-localisation approach of \cite{HaPe21}. From a computational point of view, the complexity of the approach depends strongly on the localisation, i.e.~the decay, of the operator-adapted basis functions. By using the super-localised basis representation of the wave function, speed-ups in $2$d of up to two orders of magnitude are achieved compared to the earlier work in \cite{HeWaDo22}. This huge speed-up is complemented by other practical improvements. The operator-adapted space involves a two-level discretisation. In our implementation, a virtual mesh with a classical $\mathbb{P}^3$ finite element space is used to represent the basis functions.  This is in contrast to previous work having used localised orthogonal decomposition to compute a basis and a $\mathbb{P}^1$ representation of the basis functions. The minimisation and time-stepping algorithms are state of the art and an implementation in Julia is provided. 

Overall, our novel superlocalised wave-function approximation combines robustness to low regularity with competitive speed and geometric flexibility. This is in contrast to the competing approach of Fourier-based methods, whose efficiency and optimality depend strongly on the smoothness of the potential and the wave solution. To highlight this difference, several comparisons with the popular GPELab \cite{gpelab,GPELabII} are presented. For a fast rotating large BEC, i.e.~with strong nonlinearity, we also compare with the highly optimised high performance computing package BEC2HPC \cite{BEC2HPC}. These comparisons show that the performance of our method is already competitive with spectral approaches for smooth problems. In the presence of rough potentials, our approach is of unprecedented efficiency in terms of accuracy per computation time.

Outline:  \normalfont  A short mathematical introduction to the equations is given in section \ref{sec:MathBackground}, our spatial discretisation is subsequently introduced in section \ref{sec:SpatialDiscretiation}, optimal convergence rates of a modified energy functional are proved for this spatial discretisation in section \ref{sec:Proof}. Numerical examples demonstrating these rates are given in section \ref{sec:Numerics}. A combined minimisation and time-dependent problem is solved in $3$d in section \ref{sec:temp} and some details of our implementation and complements to our proof are found in the Appendix.

\section{Mathematical modeling of BECs}\label{sec:MathBackground}
BECs provide a way to study quantum physics on much larger scales than, say, individual atoms. A BEC is formed when a dilute gas of bosons is cooled to near absolute zero. Dilute means that essentially only pairwise interactions occur, so that crystal formation is avoided. Mathematically, before condensation, the gas is described by the high-dimensional wave function $\Psi(x_1,\dots, x_N, t) \in L^2(\R^{3N},\mathbb{C})$, whose time evolution is subject to the linear Schrödinger equation, where $N$ is the number of bosons. Below a certain threshold temperature, however, most of the bosons condense into the same quantum state, whereby the wave function of the gas becomes well described by a single complex-valued wave function $u=u(x,t)$, which is governed by the cubic nonlinear Schrödinger equation (also called the Gross-Pitaevskii equation in the present context) in only $3$+$1$-dimensional space. 

In experiments with BECs, it is common to first create a BEC under the influence of a trapping potential, then perturb the BEC and study its dynamics. Mathematically, this experimental setup translates into two problems of different flavor, namely a nonlinear constrained minimisation or a nonlinear eigenvector problem, to compute a steady state that will serve as the initial state for the dynamical simulation in a perturbed configuration. 

More specifically, we will consider the problem of minimising energy
\begin{align}\label{Energy}
	E(v)= \int_{\R^3} \frac{1}{2}|\nabla v|^2+V|v|^2+v \overline{\big(\boldsymbol{\Omega}\cdot(\mathbf{x}\times- \ci \nabla)\big)v}+\frac{\beta}{2}|v|^4 \dx
	\end{align}
within an appropriate class of wave functions $\mathbb{S}\subset L^2(\R^3)= L^2(\R^3, \mathbb{C})$ subject to a unit mass constraint. 
In this formulation, $\beta$ is proportional to the number of Bosons and their interaction strength (scattering length), $V$ is a trapping potential, $\ci = \sqrt{-1}$ is the imaginary unit, $\overline{v}$ denotes the complex conjugate of $v$ and we consider the possibility of a rotating BEC whose rotation, i.e.~direction and angular velocity, is given by $\boldsymbol{\Omega}\in \R^3$. Note that the mathematical formulation is in the rotating frame of reference.
Without loss of generality, we can assume that $\boldsymbol{\Omega}=(0,0,\Omega)$
for some scalar parameter $\Omega\geq 0$. 

While the mathematical problem is typically phrased in full space $\R^3$, practical computations are restricted to some sufficiently large but bounded domain $\D \subset \R^d$ (for $d = 1, 2, 3$) which is assumed convex with a polyhedral boundary.  In fact, $\D$ need seldom be large since imposing only $\lim_{|x|\rightarrow\infty} V(x)=\infty$ on admissible potentials leads to the conclusion that any minimiser must decay exponentially fast \cite[Thm. 2.5]{BaoNum} and for harmonic trapping potentials the decay is in fact similar to that of a Gaussian. On $\D$, the Sobolev space of complex-valued, weakly differentiable functions with $L^2$-integrable partial derivatives is denoted as usual by $H^1 (\D) := H^1(\D, \mathbb{C})$ and its subspace of functions with zero trace on $\partial\D$ by $H^1_0 (\D)$. 
 The class of admissible wave functions for the minimisation problem is therefore given by
\begin{equation*}
\mathbb{S} \coloneqq \{u\in H^1_0(\D) : \|u\| = 1 \},
\end{equation*}
where $\|\bullet\|$ denotes the $L^2(\D)$ norm. 
The minimisation problem then becomes
\begin{align} \label{MinProblem3}
&\min_{v\in \mathbb{S}} \int_\D \tfrac{1}{2}|\nabla v|^2+V|v|^2+v\Omega \overline{ L_zv} +\frac{\beta}{2}|v|^4 \dx,
\end{align}
where $L_z = -\ci(x\partial_y-y\partial_x)$ is the z-component of the angular momentum. The corresponding Euler-Lagrange equation is 
\begin{align}\label{NEVP}
\lambda^0 u^0 = -\tfrac{1}{2}\triangle u^0 + V u^0+\ci \omega L_zu^0+ \beta |u^0|^2 u^0,
\end{align}
where equality is to be understood in a sufficiently weak sense. This is an eigenvalue problem for a nonlinear partial differential operator, which in numerical linear algebra and scientific computing is called a nonlinear eigenvector problem. 
 
For $\Omega = 0$, the problem of finding the normalised eigenpair $(\lambda^0,u^0)\in \R\times \mathbb{S}$ satisfying \eqref{NEVP} with minimal $\lambda^0$ and finding the (global) minimiser of \eqref{MinProblem3} are in fact  equivalent \cite{CCM10}. However, this does not hold in the case $\Omega> 0$ and indeed it is possible to find $(u^0,\lambda^0)$ and $(u_\text{\tiny gs},\lambda_\text{\tiny gs})$ satisfying Eq.~\eqref{NEVP} such that $E(u_\text{\tiny gs})<E(u^0)$ but $\lambda^0 < \lambda_\text{\tiny gs}$. A striking such example is given in Section~\ref{sec:FastRotation}. We note that in the absence of a rotational term there is a unique, real and non-negative (positive in the interior) minimiser, the argument is classical and can, for example, be found in \cite[Appendix] {CCM10}. However, if $\Omega >0$, the additional assumption that there is an $\epsilon >0$ such that 
\begin{align*}
	V(x)-\frac{1+\epsilon}{4}\Omega^2|x|^2\geq 0, 
\end{align*}
is required for $E$ to be positive, coercive and weakly lower semi-continuous, thus guaranteeing existence of a minimal energy \cite{Aft06}. In this setting uniqueness is much harder to establish. If $\Omega$ is less than some critical value, uniqueness can be recovered up to a complex shift and, in addition,  up to rotation if~$V$ is rotationally invariant. Interestingly, for a certain critical rotational speed, uniqueness can ultimately be lost in terms of the density $|u|^2$ even up to rotation \cite{BaoMarkowich}.

Once the condensate has formed under the influence of the trapping potential $V$, one may perturb it  by, e.~g., changing the potential to $\hat V$. The subsequent dynamics are governed by the Schrödinger equation:
\begin{align*}
\langle \ci \partial_t u,v\rangle =  \tfrac{1}{2} \langle E'(u),v\rangle \quad \forall v\in H^1_0(\Omega),
\end{align*}
where $E'(u)$ denotes the Fréchet derivative of $E$. The $L^2(\D)$ inner product is denoted $\langle v,w\rangle :=\int_\D v(x)\overline{w(x)}\dx$. We restrict our attention to studying the dynamics without external rotation, accordingly  the following time-dependent problem is considered: Given $ u(x,0) = u_0(x) \in H^2(\D)\cap H^1_0(\D)$, find $
	 u\in C([0,T],H^1_0(\D)) $ and $\partial_t u \in C([0,T],H^{-1}(\D))$ such that 
\begin{align}\label{GPE}
& \langle \ci \partial_t u,v\rangle_{H^{-1},H^1_0} =  \tfrac{1}{2}\langle \nabla u,\nabla v\rangle  +\langle \hat Vu+ \beta|u|^2 u,v\rangle \quad   
\end{align}
for all $v\in H^1_0(\D)$ and $t\in [0,T]$.
The assumption that $u_0 \in H^2(\D)\cap H^1_0(\D)$ may seem restrictive but is consistent with $u_0$ being an energy minimiser in the sense previously described. This time-dependent problem is locally well-posed in the sense that, on some time interval that may depend on $\|u_0\|_{H^1}$, there exists a unique solution depending continuously on the initial datum \cite[Theorem 3.3.9, Corollary 3.3.11]{Cazenave}, where we point out that the corollary enters into effect due to the regularity of the initial value, i.e., in the terms of the cited reference, $g(u)|_{t=0} =Vu_0+\beta|u_0|^2u_0 \in L^2(\D)$ provided $V\in L^2(\D)$.  
 By testing the time-dependent Eq.~\eqref{GPE} with $v = u$ and considering the imaginary part we find that the mass is conserved, 
	\begin{align*}
	\|u(t)\| = \|u_0\|.
	\end{align*}
	In a similar fashion, testing with $v = \partial_t u $ and considering the real part, formally leads to the conclusion that energy is conserved, i.e.,
	\begin{align*}
	E(u(t))= E(u_0).
	\end{align*}
	Though, a priori, this last argument assumes $\partial_t u \in L^2(\D)$, it does in fact hold in our setting as soon as $u_0\in H^1_0(\D)$ \cite[Theorem 3.3.9]{Cazenave}.
	These two conservation laws are the only known to hold in all dimensions $d\leq 3$ and in the presence of potential terms. There are however more known conservation laws in less general settings.

\section{Spatial discretisation}\label{sec:SpatialDiscretiation}
In this section we discuss the spatial discretisation for the nonlinear minimisation problem~\eqref{MinProblem3} as well as for the time-dependent problem~\eqref{GPE}. As a starting point we introduce a quasi-uniform simplicial mesh on the convex and polyhedral domain ~$\D$. The simplicial subdivision is denoted $\mathcal{T}_H$ so that $\overline{\D}= \bigcup_{T\in\mathcal{T}_H}T$ and the parameter $H$ denotes the mesh size. For an efficient implementation we will use a Cartesian grid to define the simplicial subdivision, however the method and its numerical analysis are not restricted to such a structured grid. The details of the specific mesh is therefore found in the Appendix. Given $k \in \N$ and a mesh $\mathcal{T}_H$, the finite element space $\mathbb{P}_H^k \subset H^1(\D)$ is defined by
\begin{equation*}
\mathbb{P}_H^k \coloneqq \mathbb{P}_H^k(\mathcal{T}_H) \coloneqq \big\{ v_H \in C(\overline{\D}) \,\big|\, v_H|_T \text{ is a polynomial of degree}\leq k \text{ for all } T\in \mathcal{T}_H\big\}.
\end{equation*}
For the subspace of $\mathbb{P}_H^k$-functions vanishing at a subset of boundary edges or faces $\Gamma \subset \partial \D$, we write $\mathbb{P}_{H,\Gamma}^k$.


\subsection{Ideal OD-spaces}
A natural goal in spatial discretisation is to achieve high accuracy at low resolution and low computational cost, especially for problems with multiscale or more general features of low regularity. Consider the abstract setting of a variational equation that, given $f\in H^2(\D)\cap H^1_0(\D)$, seeks $u\in H^1_0(\D)$ such that
\begin{align}\label{vareq}
a(u,v) = \langle f,v\rangle
\end{align}
holds for all $v\in H^1_0(\D)$.
For expository purposes we also write $\langle \mathcal{A}u,v\rangle = a(u,v)$ and let~$\mathcal{A}^{-1}$ denote the solution operator $\mathcal{A}^{-1}\colon L^2(\D)\mapsto H^1_0(\D)$.
Not surprisingly, the universal approximation space~$\mathbb{P}_H^1$ generally lacks the desired properties just described. In contrast, a general formal way to obtain these properties is to consider the problem-adapted space
\begin{equation*}
	 \VOD \coloneqq \mathcal{A}^{-1}\mathbb{P}^1_H.
\end{equation*}  
To demonstrate the universal approximation properties of this space, consider the candidate approximation $u_H = \mathcal{A}^{-1}P_H f$, where $P_H$ denotes the $L^2(\D)$ projection onto~$\mathbb{P}^1_H$. For~$u_H$, the equality
\begin{align}\label{discrete_vareq}
	a(u_H,v) = a(\mathcal{A}^{-1}P_H f,v) = \langle P_H f,v\rangle
\end{align}
holds for all $v\in H^1_0(\D)$. Subtracting \eqref{discrete_vareq} from \eqref{vareq}, choosing $v= u-u_H$, and assuming the coercivity of $a(\cdot,\cdot)$ leads to
\begin{align}\label{OD-H1-est}
	c \|u-u_H\|_{H^1(\D)}^2 \leq a(u-u_H,u-u_H) &= \langle f-P_H f,u-u_H\rangle \nonumber  \\
	& = \langle f-P_H f,u-u_H-P_H(u-u_H)\rangle \nonumber \\
	& \leq CH^2 \|f\|_{H^2(\D)} H \|u-u_H\|_{H^1(\D)}.
\end{align}
This then implies, via Céa's lemma, the error estimate
\begin{align}
\min_{v\in \mathcal{A}^{-1}\mathbb{P}_H^1}\|u-v\|_{H^1(\D)}\lesssim	\|u-u_H\|_{H^1(\D)} \lesssim H^3 \|f\|_{H^2(\D)}.
\end{align}
Note that the error estimate does not depend on the possible oscillatory nature of $a$, nor on its regularity (other than its coercivity).
To derive $L^2$ estimates, we use the following characterisation of the operator-adapted space.
	The space $\VOD = \mathcal{A}^{-1}\mathbb{P}_H^1$ equals the $a$-orthogonal complement 
	of the kernel of $P_H$ in $H^1_0(\D)$, i.e.,  
\begin{equation}\label{eq:charac_OD} \mathcal{A}^{-1}\mathbb{P}^1_H = \ker(P_H)^{\perp_a}.
\end{equation}
A proof can be found in~\cite[Rem.~3.6 \&~3.7]{ActaNumericaLOD}.
%
Thus we also have the ideal splitting $H^1_0(\D) = \mathcal{A}^{-1}\mathbb{P}_H^1\oplus \ker(P_H)$.  For this reason, the space~$\VOD = \mathcal{A}^{-1}\mathbb{P}_H^1$ is also called OD space, where OD is short for Orthogonal Decomposition. With Eq.~\eqref{eq:charac_OD}, it is straightforward to derive $L^2$ error estimates for the solution of the variational equation in the OD space, denoted $u_\OD$. Since $u-u_\OD \in \ker(P_H)$, there holds
\begin{align}\label{OD-L2-est}
	\|u-u_\OD\| = \|u-u_\OD -P_H(u-u_\OD)\| \leq CH \|u-u_\OD\|_{H^1}\lesssim H^4. 
\end{align}
Now the question arises, what is $\mathcal{A}$ in our context? A priori, $a$ can be any linear operator associated with the equations \eqref{NEVP} and \eqref{GPE}. For this cubic nonlinear Schrödinger equation, the choice of $\mathcal{A} = -	\triangle + V_d$, where $V_d$ is the low regularity part of the potential, has proven sufficient both theoretically and numerically, at least in the absence of rotational terms. Thus, for sufficiently smooth $V$, $\mathcal{A} = -	\triangle$ is a good choice. Our numerical experiments show that this is also true for rotating BECs. As will become clear, the more precise statement is that $V_d$ is any part of the potential with regularity less than $H^2(\D)$.  As an aside, we note that it is also possible to define higher order OD spaces, see \cite{Maier21}.

\subsection{Super-localised wave function approximation}\label{SLOD}
As an approximation of a basis of the OD-space $\VOD$, we use the so-called super-localised orthogonal decomposition (SLOD)~\cite{HaPe21}. This means that, similar to the Wannier functions~\cite{Wan37,Wan62}, our approximation space is represented by problem-adapted, local responses to linear operators associated with the Eq.~\eqref{NEVP}. To represent these responses, we use classical finite element spaces. Since localisation of the basis functions is key, we define the local spaces $\mathbb{P}_H^k(\omega)$ and $\mathbb{P}_{H,\Gamma}^k(\omega)$ for any open subset $\omega \subset \D$ whose triangulation is a subset of~$\mathcal{T}_H$. More specifically, given a node~$z \in \mathcal{N}_H$, we define the $\ell$th-order patch $\omega_{\ell}(z)$ for a given $\ell=0,1,2,\ldots$ as 
 
\begin{equation*}
	\omega_\ell(z) \coloneqq \bigcup_{\substack{T \in \mathcal{T}_H,\\ T \subset B_{H(\ell+1)}(z)}} T,
\end{equation*}
where $B_{r}(z)$ denotes the ball with radius $r$ around $z$ with respect to the $\infty$-norm in $\R^d$.
We will use this notation to illustrate the SLOD for one-dimensional domains first. The ideas are then easily transferred to higher dimensions.

\subsubsection{Perfectly localised basis functions in one dimension}
Consider the differential operator $\mathcal{A}=-\partial_{xx} + V_d$ defined on a finite interval and with homogeneous Dirichlet boundary conditions. Here, $V_d$ is the non-regular part of the potential~$V$, such as a highly oscillatory or discontinuous part. For a single hat function $\Lambda_z \in \mathbb{P}_H^1$ centered at node $z$, the response or wave, for lack of a better word, of $\Lambda_z$ to the inverse operator $\mathcal{A}^{-1}$ is generally globally supported, cf. Figure~\ref{fig:SLOD_1d}. The main idea behind SLOD is to use scaled wave responses associated with neighboring nodes as destructive interference, such that the tails of the waves cancel each other in such a way that their linear combination is locally supported. These linearly combined waves then form the basis of our approximation spaces.
\begin{figure}
	\begin{tabular}{cc}
			\input{P1_OD_single} & 	\input{P1_SLOD_single}
	\end{tabular}
	\caption{Response of a hat-function (left) and localised basis function (right) for the Laplace operator in one dimension. The dashed line illustrates the scaled right-hand side.}
	\label{fig:SLOD_1d}
\end{figure}
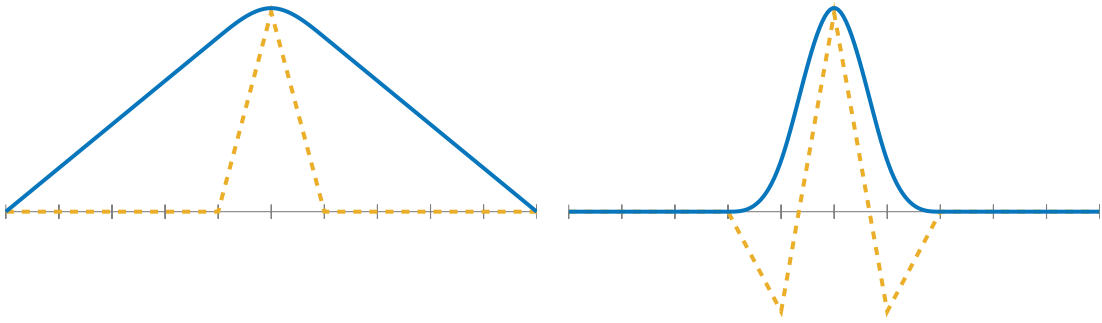
Since in one dimension there is only one direction in which the initial wave must be canceled away from the central node, two additional waves -- one to the left and one to the right -- (or one if $z \in \partial \D$) are sufficient to obtain a non-vanishing, fully localised basis function, e.g.~for $\mathcal{A} = -\partial_{xx}$ we have that $\mathcal{A}^{-1}(2\Lambda_z-(\Lambda_{z+1}+\Lambda_{z-1})) $ has local support, cf. Figure~\ref{fig:SLOD_1d}. The support of the resulting basis function is contained in a patch of order 1 around the node $z$, i.e.~$\omega_1(z)$. Note that the number of waves needed for localisation is independent of the choice of $V_d$.  

To represent the basis functions, one typically uses a finite element space~$\mathbb{P}_{h,\partial \D}^k$ with more
degrees of freedom due to a higher polynomial degree, $k\geq 1$, or due to a finer underlying mesh $h\leq H$. A higher polynomial degree can capture the smoothness of the responses, while a finer mesh can capture
information on smaller scales, e.g.~due to a highly oscillatory potential~$V_d$. However, the number of basis functions of our discretisation space is independent of its representation and is only determined by the underlying space of the right-hand sides on which $\mathcal{A}^{-1}$ operates, i.e.~in our case~$\mathbb{P}_H^1$. 
  
In higher dimensions, where there are infinitely many directions in which to cancel, it is still an open question whether perfect localisation is possible \cite{GrGrSa12}. However, we will show below that quasi-locality in the sense of extremely fast decay to zero is always possible. 

\subsubsection{Super-localisation in dimension two and three}
Before generalising to higher dimensions, let us formalize the one-dimensional problem; since we know that there exists a normalised response~$\varphi_\OD$ that is zero outside of $\omega_1 \coloneqq \omega_1(z)$, it becomes possible to compute $\varphi_\OD$ locally. For this purpose we introduce the operator $\mathcal{A}$ restricted to $H^1_0(\omega_1)$, which we denote $\mathcal{A}|_{\omega_1}$, and search for the normalised right side $p^* \in \mathbb{P}_{H,\partial \omega_1 \setminus \partial\D}^1(\omega_1)$, such that $\varphi^*_{\OD,1}=\mathcal{A}|_{\omega_1}^{-1}p^*$ minimises the conormal derivative~$-\partial_x \varphi^*_{\OD,1}$ at the local patch boundary~$\partial \omega_1 \setminus \partial \D$. Note that $-\partial_x \varphi^*_{\OD,1}=0$ holds at~$\partial \omega_1 \setminus \partial \D$ by counting the degrees of freedom, and thus we have $\varphi^*_{\OD,1} = \pm \varphi_{\OD}$ if we extend~$\varphi^*_{\OD,1}$ by zero outside of~$\omega_1$. 

Similarly in two and three dimensions, we fix $\ell\geq 1$ and consider local problems around each node $z$ on corresponding patches $\omega_\ell \coloneqq \omega_\ell(z)$. Given the restriction $\mathcal A|_{\omega_\ell}\colon H^1_0(\omega_\ell) \to H^{-1}(\omega_\ell)$ of $\mathcal A$ and the local approximation (extended by zero) $\varphi_{\OD,\ell} = \mathcal{A}|_{\omega_\ell}^{-1}p$ of $\varphi_\OD = \mathcal{A}^{-1}p$, for a right-hand side $p \in \mathbb{P}_{H,\partial \omega_\ell \setminus \partial\D}^1(\omega_\ell)$, one can show that
\begin{equation}\label{sigma}
	\| \varphi_{\OD,\ell} - \varphi_{\OD}\|_a = \sup_{v \in H^1_0(\D)\setminus \{0\}} \frac{a(\varphi_{\OD,\ell} - \varphi_{\OD},v)}{\|v\|_{a}} = \sup_{v \in H^1_0(\D)\setminus \{0\}} \frac{1}{\|v\|_{a}} \int_{\omega_\ell} -(\nabla \varphi_{\OD,\ell})\cdot n v\, \textrm{d}S 
\end{equation}
holds, see \cite[p.~5~f.]{HaPe21}. Therefore, to minimise the localisation error, we compute
\begin{equation}\label{eq:choice_SLOD}
	p^* = \underset{\substack{p \in \mathbb{P}_{H,\partial \omega_\ell \setminus \partial\D}^1(\omega_\ell)\\\text{s.t. } \|p\|_{L^2(\omega_\ell)}=1}}{\operatorname{argmin}} \| -\nabla ( \mathcal{A}|_{\omega_\ell}^{-1}p) \cdot n  \|_{L^2(\partial \omega_\ell \setminus \partial\D)}.
\end{equation}
Note that, in the original paper~\cite{HaPe21}, the authors search for a $p$ that is (almost) $L^2$-orthogonal to the space of $\mathcal{A}$-harmonic functions on $\omega_\ell$.
This leads to a singular value decomposition of a large matrix in order to capture the behaviour of the $\mathcal{A}$-harmonic functions. Our approach~\eqref{eq:choice_SLOD}, on the other hand, requires only to solve a generalised eigenvalue problem of the size of~$\operatorname{dim} \mathbb{P}_{H,\partial \omega_\ell \setminus \partial\D}^1(\omega_{\ell})$, while the localised response $\varphi_{\OD,\ell}^*$ of $p^\ast$ shows the super-exponential decaying localisation error with respect to $\ell$, which is numerically observed and justified under a  spectral geometric conjecture in~\cite[Sec.~7~f.]{HaPe21}, see Figure~\ref{fig:SLOD_decay}.  \begin{figure}
	\centering
		\includegraphics[width=0.8\linewidth]{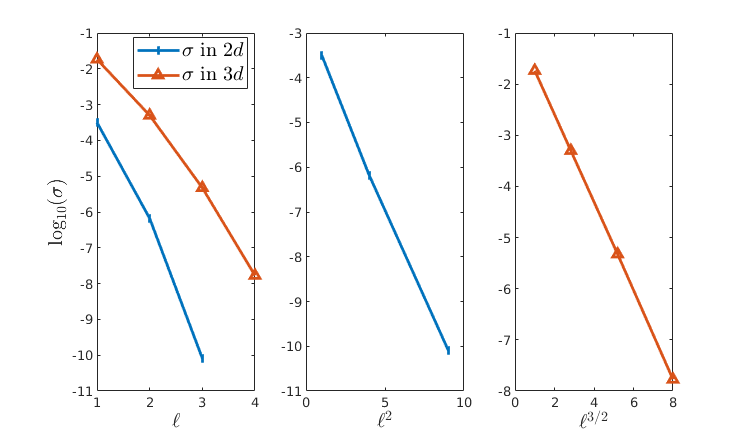}
\caption{optimised localisation error~$\sigma = \|\varphi_{\OD,\ell}-\varphi_\OD\|_a$, cf. Eq.~\eqref{sigma}, with respect to the width parameter $\ell$ of the patch. In $2$d the scaling $\sigma \sim e^{-c\ell^2}$ is observed and in $3$d  $\sigma \sim e^{-c\ell^{3/2}}$ is observed. This particular case is the canonical basis with $\mathcal{A}=-	\triangle$.} 
\label{fig:SLOD_decay}	
\end{figure}
If the patch~$\omega_\ell(z)$ intersects the boundary~$\partial \D$ for the node~$z$, then we consider a second minimisation problem by introducing an artificial trapping potential to find an optimal linear combination of $\mathcal{A}|_{\omega_\ell}^{-1}p^*_i$, $i=1,\ldots,n$ which  localises around the current node~$z$. Here, the~$p^*_i$s denote the first~$n$, pairwise $L^2$-orthogonal minimisers. The linear combination is again denoted by $p^*$. 

The basis of the SLOD space $\VSLOD$ is then given by the function $\varphi_{\OD,\ell}= \mathcal A|_{\omega_{\ell}}^{-1}p^\ast$ associated with each node~$z$ in $\mathcal{T}_H$. To represent this basis we use the space $\mathbb{P}^3_{h,\partial \D}$ with a spatial mesh width $h\leq H$. There are two main motivations for this choice of representation of the SLOD basis functions. First it is noticed that in $1$d, the canonical SLOD-basis functions, i.e., with $\mathcal{A}=-\partial_{xx}$, coincide with cubic B-splines and are thus exactly represented by piecewise cubic polynomials on the same mesh. In higher dimensions it no longer holds that $\mathbb{P}^3_H$-functions exactly represent the SLOD basis functions. However, for sufficiently smooth problems, there is still an indication that, while not an exact representation, no order of accuracy is lost when approximating the canonical SLOD basis functions with $\mathbb{P}^3_H$ functions on the same mesh as the $\mathbb{P}^1_H$-functions. Namely, considering the NEVP Eq. \eqref{NEVP},  the eigenvalue is expected to converge, for sufficiently smooth problems, with $\mathcal{O}(H^6)$ in the $\mathbb{P}^3_H$-space (see \cite{henning2023discrete} for a precise statement in arbitrarily high $\mathbb{P}_H^k$-spaces and \cite{CCM10} for $k\leq 2$). Therefore, the best approximation in the underlying $\mathbb{P}^3_H$-space allows for optimal convergence in the canonical OD space based on $\mathbb{P}^1_H$-functions. In cases of less regularity, e.g., when the potential is included into the construction of the space, a refined mesh and its  $\mathbb{P}^3_{h}$-space must be used to represent the SLOD basis functions.

\section{A modified energy minimisation problem}\label{sec:Proof}
We now turn to the analysis of a slightly modified energy minimisation problem in the spaces introduced in Section~\ref{sec:SpatialDiscretiation}, with special attention to potentials of low regularity. The modification is introduced as a means to speed up the computations, the details of the speed-up are outlined in the Appendix \ref{assembly_nonlinear}. Given the ideal OD-space $\VOD$, we introduce the modified energy functional,
\begin{align}\label{modified_E}
	\tilde{E}(u) & = \int \frac{1}{2}|\nabla u|^2 + V|u|^2+\frac{\beta}{2}P_\OD(|u|^2)|u|^2\dx \nonumber \\
	& = \int \frac{1}{2}|\nabla u|^2 + V|u|^2+\frac{\beta}{2}P_\OD(|u|^2)^2\dx,
\end{align}
where $P_\OD\colon L^2(D)\mapsto \VOD$ denotes the $L^2$-projection. With this, we consider minimising with respect to $\tilde{E}$, but evaluating $E$ at the minimiser of $\tilde{E}´$, i.e., 
\begin{align*}
E\bigg(\underset{v\in \mathbb{S}\cap  \VOD}{\text{argmin}}\ \tilde{E}(v)\bigg)
\end{align*}
In this section, we prove optimal convergence rates of this energy, without rotation, under the weak assumption that $V\in L^{2+\sigma}(\Omega)$, where $\sigma = 0$ in $1$d, $2$d, but $\sigma >0$ in $3$d. By optimal we mean $\mathcal{O}(H^6)$ in terms of minimum energy and $\mathcal{O}(H^3)$ as measured in the $a$-norm or the equivalent $H^1$-norm. In contrast, and rather remarkably, we show that in low regularity regimes the modified energy converges only as $\mathcal{O}(H^4)$. For the ease of presentation we introduce the following notation:
\begin{gather*}
a(u,v) = \int \frac{1}{2}\nabla u \nabla v+ V uv \dx, \qquad	\|u\|_a^2 = a(u,u), 
\\ 
\tilde{u}^0_\OD = \underset{u\in \mathbb{S}\cap\VOD}{\text{argmin}}\ \tilde{E}(v), \qquad u^0_\OD = \underset{u\in\mathbb{S}\cap \VOD}{\text{argmin}}\ E(v), \qquad u^0 = \underset{u\in\mathbb{S}\cap H^1_0(\D)}{\text{argmin}}\ E(v).
\end{gather*} 
The operator~$\mathcal{A}\colon H^1_0(\D) \to H^{-1}(\D)$ is defined with respect to the bilinear form~$a$. For notational brevity we shall prove the convergence in the OD space~$\VOD = \mathcal{A}^{-1}\mathbb{P}^1_H$ using the inner product $a$. 
However, it is emphasised that it is sufficient to include only the low regularity part of the potential~$V$ in the construction of the OD space. That is, whenever the potential splits into two contributions $V = V_{\text{s}}+V_{\text{d}}$, $V_{\text{s}}$ being at least in $H^2(\D)$, $V_{\text{d}}$ being in $L^{2+\sigma}(\Omega)$, optimal convergence rates are obtained in the OD space defined by the inner product $a_{\text{d}}(u,v) = \tfrac{1}{2}(\nabla u,\nabla v)+(V_{\text{d}}u,v)$ and its associated operator $\mathcal{A}_{\text{d}}$, see Remark~\ref{rem:a_with_Vd}. This can also be of computational importance if $V_{\text{d}}$ is periodic but $V_{\text{s}}$ is not. Similarly, whenever $V$ is smooth, the canonical OD space, i.e., $(-	\triangle)^{-1}\mathbb{P}^1_H$, achieves an optimal order of convergence.

\subsection{Error estimates for OD spaces} \label{truncation-error}

The high-level idea behind our proof is based on the observation that 
the first variation of the energy at $\tilde{u}^0_\OD$, in the direction of $u^0_\OD-\tilde{u}^0_\OD$ behaves like ${E}'(\tilde{u}^0_\OD)[u^0_\OD-\tilde{u}^0_\OD] \sim H^3\| u^0_\OD-\tilde{u}^0_\OD\|_a$, at the same time at the minimiser we have 
$|E'(u^0_\OD)[u^0_\OD-\tilde{u}^0_\OD]| = \mathcal{O}(\|\tilde{u}^0_\OD-u^0_\OD\|_a^2)$, A clever use of the properties of the second variation will then allow us to conclude that $\|u^0_\OD-\tilde{u}^0_\OD\|_a\lesssim H^3$ and subsequently that $E(\tilde{u}^0_\OD) = E(u^0_\OD) +\mathcal{O}(\|u^0_\OD-\tilde{u}^0_\OD\|_a^2)$.

We begin by recalling some known results about the minimiser of $E$, $u^0_\OD$.

\begin{theorem}[Error estimates for $u^0_\OD$]\label{th:est_VOD_classic}
	Let $u^0_\OD \in \VOD$ be an energy minimiser of $E$ with $(u^0_\OD,u^0)\geq0$. 
	If $H$ is small enough, then
	\begin{equation*}
		\|u^0_\OD-u^0\|_{H^1(\D)} \lesssim H^3, \qquad |E(u^0_\OD) - E(u^0)| \lesssim H^6,  \qquad |\lambda^0_\OD - \lambda^0| \lesssim H^3,
	\end{equation*}
	where the eigenvalue $\lambda^0_\OD$ is given by $2E(u^0_\OD)+\frac \beta 2 \|u^0_\OD\|_{L^4(\D)}^4$. The constants only depend on the data and the mesh regularity of the triangulation.
\end{theorem}
\begin{proof}
	The estimates of the $H^1$ and energy errors are proved in \cite[Sec.~4]{HeP21}. Consequently, we have by \cite[Th.~1]{CCM10}, that
	\begin{equation*}
		|\lambda^0_\OD - \lambda^0| \lesssim \|u^0_\OD - u^0\|_{H^1}^2 + \|u^0_\OD - u^0\|_{L^2} \leq \|u^0_\OD - u^0\|_{H^1}^2 + \|u^0_\OD - u^0\|_{H^1} \lesssim H^3. \qedhere
	\end{equation*}
\end{proof}

\begin{remark}
	For $V \in L^\infty(\D)$, the estimate of the error of the eigenvalues can be improved to sixth order \cite[Prop.~4.6]{HeP21}.
\end{remark}

The next step is to investigate the smoothness of functions in~$\VOD$ in general, and of $u^0_\OD$ in particular. For this, we need that the $a$-projection to $\VOD$ is $L^2(\D)$-stable. 

\begin{lemma}[$H^2(\D)$-regularity of $\VOD$]\label{lem:VOD_subset_H2} If $V \in L^{2+\sigma}(\D)$ is satisfied with $\sigma =0$ for $d=1,2$ and $\sigma >0$ for $d=3$, then $\VOD \subset H^1_0(\D) \cap H^2(\D)$ holds.
\end{lemma}
\begin{proof}
	Let $v_\OD \in \VOD$ be arbitrary. By definition, there exists a $p \in \mathbb{P}^1_H \subset H^1(\D) $ such that $-\frac 1 2 	\triangle v_\OD + V v_\OD = p$ in the weak sense in $H^1_0(\D)$. In particular, $v_\OD \in H^1_0(\D) \hookrightarrow L^6(\D)$ holds and, hence, $-\frac 1 2 	\triangle v_\OD = p - Vv_\OD \eqqcolon g \in L^r(\D)$ with $r = 3/2+9\sigma/(16+2\sigma) $ from Hölder's inequality. Notably, $r>d/2$ holds which by \cite[Sec.~7, Th.~1.5]{Nec12} implies that the solution~$v_\OD$ is bounded almost everywhere in $\D$. Therefore, $Vv_\OD$ and, hence, the right-hand side~$g$ are $L^2(\D)$-functions. This implies $v_\OD \in H^2(\D)$ by standard smoothness results, see, e.g., \cite[Th.~9.1.22]{Hac03}.
\end{proof}

\begin{remark}\label{rem:VOD_subset_H2_bound}
	With the cited references in the proof of Lemma~\ref{lem:VOD_subset_H2}, we have 
	\begin{equation*}
		\|v_\OD\|_{H^2} \lesssim \|\tfrac 1 2 	\triangle v_\OD\| 
		\leq \|p\|+\|V\|\,\|v_\OD\|_{L^\infty} \lesssim \|p\|+\|V\|(\|p-Vv_\OD\|_{L^r}+\|v_\OD\|_{H^1}) \lesssim \|p\|
	\end{equation*}
	with the notation of the above proof.
\end{remark}

%

\begin{lemma}[$L^2(\D)$-stability of $A_\OD$] \label{lem:L2_bound_A_OD}
	The $a$-projection $A_\OD$ from $H^1_0(\D)$ to $\VOD$ is $L^2$-stable, i.e., $\|A_\OD v\|_{L^2(\D)} \lesssim \|v\|_{L^2(\D)}$ for every $v\in H^1_0(\D)$, independent of $H$.
\end{lemma}

\begin{proof}
	Let $P_H$ be the $L^2$-projection onto $\mathbb{P}^1_H$. Furthermore, we define the corrector $C\colon H^1_0(\D)\ni u \to w \in \mathcal{W} \coloneqq \ker P_H \cap H^1_0(\D)$ via the (partial) solution~$(w,\mu) \in H^1_0(\D) \times \mathbb{P}^1_H$ of the saddle-point problem
	\begin{equation*}
		a(w,\hat v)+\int_D \nu P_H w - \mu P_H \hat v \dx = a(u,\hat v)
	\end{equation*}
	for every $\hat v\in H^1_0(\D)$ and $\nu \in \mathbb{P}^1_H$.
	Then it is well-known that $\VOD = (\operatorname{id} - C)H^1_0(\D)$ holds and $C$ is a projection onto $\mathcal W$, as well as that $\mathcal V_\OD$ and $\mathcal W$ are $a$-orthogonal complements in $H^1_0(\D)$, 
	see \cite[Sec.~3.2]{ActaNumericaLOD} and \eqref{eq:charac_OD}. In particular, this implies 
	\begin{equation*}
		A_\OD v = (\operatorname{id} - C) v	= (\operatorname{id} - C)P_H v + (\operatorname{id} - C)(\operatorname{id} - P_H) v= (\operatorname{id} - C)P_H v
	\end{equation*}
	for every $v \in H^1_0(\D)$. Here, in the last equality, we used that $(\operatorname{id} - P_H) v$ is an element of $\mathcal W$. Finally, we have
	\begin{align*}
		\|A_\OD v\| = \|(\operatorname{id} - C)P_H v\|
		& \leq \|P_H v\| + \|(\operatorname{id} - P_H)CP_H v\|\\
		& \lesssim \|v\| + H \|CP_H v\|_{H^1}\\
		&\lesssim \|v\| + H \|P_H v\|_{H^1} \lesssim \|v\|
	\end{align*}
	by the inverse estimate, see, e.g., \cite[Ch.~II.6.8]{Bra07}
\end{proof}

\begin{lemma}[$H^2(\D)$-boundedness of $u^0_\OD$]\label{lem:H2_bound_u0_OD}
	The $H^2(\D)$-norm of an energy minimiser~$u^0_\OD$ in the OD-space~$\VOD$ is bounded independently of $H$.
\end{lemma}	

\begin{proof}
	By the density of $H^1_0(\D)$ into $L^2(\D)$, $H^1_0(\D) \hookrightarrow L^6(\D)$, Lemma~\ref{lem:L2_bound_A_OD}, and the definition of the eigenvalue $\lambda_\OD^0$, we have the estimate
	\begin{align*}
		\| - 	\triangle u^0_\OD + V u^0_\OD\| =& \sup_{v\in H^1_0(\D)\setminus \{0\}} \frac 1 {\|v\|} a(u^0_\OD,v)\\
		=& \sup_{v\in H^1_0(\D)\setminus \{0\}} \frac 1 {\|v\|} a(u^0_\OD,A_\OD v)\\
		\leq & \sup_{v\in H^1_0(\D)\setminus \{0\}} \frac 1 {\|v\|} \|\lambda^0_\OD u^0_\OD - 2 \beta |u^0_\OD|^2 u^0_\OD\|\,\|A_\OD v\| \\
		\lesssim&\, |\lambda^0_\OD| + 2\beta \|u^0_\OD\|^3_{H^1}\\
		\leq &\, 5 E(u_\OD^0) + 4 \sqrt 2\beta E^{3/2}(u_\OD^0)
	\end{align*}
	with a constant only depending on $\D$ and the shape regularity of the triangulation.
	Since~$u_\OD^0$ is the minimiser of~$E$ for $L^2$-normalised functions in $\VOD$, $E(u_\OD^0)$ in the right-hand side can be replaced by~$E(v_\OD)$ for any $v_\OD \in \VOD$ with $\|v_\OD\|=1$. 
	Hence, $- 	\triangle u^0_\OD + V u^0_\OD$ is bounded in $L^2$. The $H^2$-bound then follows by the lines of Lemma~\ref{lem:VOD_subset_H2} and Remark~\ref{rem:VOD_subset_H2_bound}.
\end{proof}
Naturally, we will need to estimate the effect of replacing $|u_\OD|^2 $ with $P_\OD(|u_\OD|^2)$. 
\begin{lemma}\label{lem:density_replacement} Let $V\in L^2(\D)$, then the estimates
	\begin{subequations}
		\begin{align}
			\|\,|u|^2 - P_\OD(|u|^2)\|_{L^2(\D)} &\lesssim H^2\|u\|^2_{H^2(\D)}, \label{eqn:est_u2}\\
			\|u\eta - P_\OD(u\eta)\|_{L^2(\D)} &\lesssim H\|u\|_{H^2(\D)}\|\eta\|_a \label{eqn:est_ueta}
		\end{align}
	\end{subequations}
	hold for every $u \in H^1_0(\D) \cap H^2(\D)$ and $\eta \in H^1_0(\D)$ with constants only depending on the dimension~$d$, the domain~$\D$, the shape regularity of the triangulation, and the potential~$V$.
\end{lemma}

\begin{proof}
	Let us first estimate~\eqref{eqn:est_u2} and then prove~\eqref{eqn:est_ueta}. \medskip
	
	\emph{Estimate~\eqref{eqn:est_u2}}: To ease notation, we introduce the density $\rho=|u|^2$. To estimate how well $P_\OD(\rho)$ approximates $\rho$, consider first the $a$-projection of $\rho$  onto $\VOD$, i.e., $A_\OD \rho$.
	We shall now use the properties of the OD-space~$\VOD$ to demonstrate the bound
	\begin{align}\label{ProjectionError2}
		\|\rho-A_\OD\rho\|_{L^2} \lesssim H^2\|f\|_{L^2},
	\end{align} 
	where $f$ denotes $-\tfrac{1}{2}\triangle \rho+V\rho$.
	Using the coercivity of $a(\cdot,\cdot)$ and the fact that  $\rho-A_\OD \rho$ is in $\text{ker}(P_H)$, we find:
	\begin{align*}
		\|\rho-A_\OD\rho\|_{H^1}^2&\lesssim a(\rho-A_\OD\rho,\rho-A_\OD\rho) = \langle f,\rho-A_\OD\rho \rangle\\
		&= \langle f,\rho-A_\OD\rho-P_H(\rho-A_\OD\rho)\rangle	\\
		&\lesssim  \|f\|\ H\|\rho-A_\OD\rho\|_{H^1},
	\end{align*}
	wherefore $\|\rho-A_\OD(\rho)\|_{H^1}\lesssim H\|f\|$. For the $L^2$-estimate we notice,
	\[\|\rho-A_\OD\rho\| = \| \rho-A_\OD\rho-P_H(\rho-A_\OD\rho)\| \lesssim H\|\rho-A_\OD\rho\|_{H^1}, \]
	wherefore $\|\rho-A_\OD\rho\|\lesssim H^2\|f\|$.	As for the regularity of the right-hand side, it is left to show that 
	\begin{align*}
		f = -\frac{1}{2}\triangle \rho +V\rho = -(u \triangle u + (\nabla u)^2)+ Vu^2
	\end{align*}
	is an element of $L^2(\D)$.
	Using the Sobolev embeddings $H^1(\D)\hookrightarrow L^6(\D)$ and $H^2(\D) \hookrightarrow L^\infty(\D)$ for $d\leq 3$, we have 
	$\|\nabla u\|_{L^4} + \|u\|_{L^\infty} \lesssim \|u\|_{H^2}$. This yields the bound
	\begin{equation*}
		\| f\| \leq \|u\|_{L^\infty}\|u\|_{H^2} + \|\nabla u\|_{L^4}^2 + \|V\|\|u\|_{L^\infty}^2 \lesssim (1+\|V\|)\|u\|_{H^2}^2.
	\end{equation*}
	In sum, this implies the stated estimate $\| \rho - P_\OD (\rho)\| \leq \| \rho - A_\OD (\rho)\| \lesssim H^2 \|u\|_{H^2}^2$. \medskip
	
	\emph{Estimate~\eqref{eqn:est_ueta}}: For the error  $\| P_\OD(u\eta)-u\eta\|$, we observe
	\allowdisplaybreaks
	\begin{equation}
		\label{estimate_u_eta}
		\begin{aligned}
			&\|P_\OD(u\eta)-u\eta\|^2\\
			\leq\, & \|A_\OD(u\eta)-u\eta\|^2 \\
			\leq\, & \| A_\OD(u\eta)-u\eta-P_H(A_\OD(u\eta)-u\eta)\|^2\\
			\lesssim\, &  H^2 \|A_\OD(u\eta)-u\eta\|_{H^1}^2 \leq H^2\|A_\OD(u\eta)-u\eta\|_{a}^2 \leq H^2 \|u \eta \|_{a}^2  \\
			=\, & H^2\big( \langle\eta^2, \nabla u \cdot \nabla u \rangle + 2 \langle\eta u \nabla u, \nabla \eta \rangle + \langle u^2 \nabla \eta, \nabla \eta \rangle + \langle u^2 V\eta, \eta \rangle \big)\\
			\leq\, & H^2 \big( \|\eta\|_{L^4}^2 \|\nabla u\|_{L^4}^2 + 2 \|u\|_{L^{\infty}} \|\nabla u\|_{L^4} \|\eta\|_{L^4} \|\eta\|_{H^1} + \|u\|_{L^\infty}^2\|\eta\|_{H^1}^2 + \|u\|_{L^\infty}^2 \|V\|\, \| \eta\|_{L^4}^2\big)\\
			\lesssim\,& H^2 \|u\|_{H^2}^2 \|\eta\|_{H^1}^2  \leq H^2 \|u\|_{H^2}^2 \|\eta\|_{a}^2.
		\end{aligned}
	\end{equation}
	Here, we used Sobolev embeddings similar to the ones in the proof of~\eqref{eqn:est_u2}.
\end{proof}

We note that, a priori, Lemma \ref{lem:H2_bound_u0_OD} does not translate to an $H^2(\D)$-bound of $\tilde{u}^0_\OD$ since $\|P_\OD|\tilde{u}^0_\OD|^2\tilde{u}^0_\OD\|$ need not be bounded by $\|\tilde{u}^0_\OD\|_{L^6}^3$. It is, however, possible to obtain a bound by complementing the argument of  Lemma \ref{lem:H2_bound_u0_OD}. 
\begin{lemma}[$H^2(\D)$-boundedness of $\tilde u^0_\OD$]\label{H^2-tilde}
	Let $\tilde{u}_\OD^0$ be a minimiser of $\tilde E$. Then $\|\tilde u_\OD^0\|_{H^2(\D)}\leq C$ holds with a constant $C$ independent of $H$. 
\end{lemma}
\begin{proof}
	For the sake of readability, we refer to the Appendix~\ref{app:H^2-tilde} for a proof.
\end{proof}

Before  proving the optimal order of convergence as measured in energy, eigenvalue and $a$-norm, some sub-optimal estimates are needed.
\begin{lemma}[Crude convergence estimate I] \label{lem:crude_est} 
	Let $u^0_\OD, \tilde u^0_\OD \in \VOD$ be energy minimisers of $E$ and $\tilde E$, respectively, with $(u^0_\OD,u^0)\geq0$ and $(\tilde u^0_\OD,u^0)\geq0$. 
	If $H$ is small enough, then
	\begin{equation}\label{eq:crude_est}
		\|u^0_\OD-\tilde{u}^0_\OD\|_a \lesssim H^2, 
	\end{equation}
	where the constant only depends on the data and the mesh regularity of the triangulation.
\end{lemma}

\begin{proof}
	By Theorem~\ref{th:est_VOD_classic} and the triangle inequality, it is enough to prove $\|\tilde{u}^0_\OD-u^0\|_{H^1} \lesssim H^2$. For this, we use 
	\begin{equation}\label{proof:crude_est_1}
		\|\tilde{u}^0_\OD-u^0\|_{H^1}^2 \lesssim E(\tilde{u}^0_\OD)-E(u^0),
	\end{equation}
	which follows the lines of \cite[p.~96~f.]{CCM10}.
	
	To estimate the difference of the energies, we notice that for every $v\in H^2\cap H^1_0$, we have 
	\begin{equation}\label{eq:tildeE<=E}
		\tilde{E}(v) = E(v) - \frac \beta 2 \||v|^2 - P_\OD(|v|^2)\|^2.
	\end{equation}
	On one hand, this implies $\tilde{E}(v) \leq E(v)$ and, hence, the order 
	\begin{equation}\label{eq:order_energies}
		\tilde{E}(\tilde{u}^0_\OD)\leq \tilde{E}(u^0_\OD)\leq E(u^0_\OD)\leq E(\tilde{u}^0_\OD).
	\end{equation}
	With estimate~\eqref{eqn:est_u2} and Lemma~\ref{H^2-tilde}, on the other hand, the bound
	\begin{equation}\label{proof:crude_est_2}
		E(\tilde{u}^0_\OD) - E(u^0_\OD) \leq E(\tilde{u}^0_\OD) - \tilde{E}(\tilde{u}^0_\OD) = \frac \beta 2 \||\tilde{u}^0_\OD|^2 - P_\OD(|\tilde{u}^0_\OD|^2)\|^2 \lesssim H^4\|\tilde{u}^0_\OD\|_{H^2}^4 \lesssim H^4
	\end{equation}
	holds. 
	In summary, estimates~\eqref{proof:crude_est_1}, \eqref{proof:crude_est_2}, and Theorem~\ref{th:est_VOD_classic} prove
	\begin{equation*}
		\|\tilde{u}^0_\OD-u^0\|_{H^1}^2 \lesssim  E(\tilde{u}^0_\OD)-E(u^0_\OD) + E(u^0_\OD) -E(u^0) \lesssim H^4 + H^6. \qedhere
	\end{equation*}
\end{proof}

\begin{lemma}[Crude convergence estimate II]\label{lem:est_EW}
	Let $\tilde \lambda^0_\OD \coloneqq 2\tilde{E}(\tilde{u}^0_\OD)+\frac \beta 2 \| P_\OD(|\tilde{u}^0_\OD|^2)\|^2$ with~$\tilde u^0_\OD$ from Lemma~\ref{lem:crude_est} and $\lambda^0_\OD$ be defined as in Theorem~\ref{th:est_VOD_classic}, then the estimate 
	\begin{equation*}
		|\lambda_\OD^0 - \tilde{\lambda}_\OD^0 | \lesssim H^{2}
	\end{equation*}
	holds if $H$ is small enough. The constant  only depends on the data and the mesh regularity of the triangulation.
\end{lemma}
\begin{proof}
	By the definition of the eigenvalues, the difference is given by
	\begin{equation*}
		|\lambda_\OD^0 - \tilde{\lambda}_\OD^0 | \leq 2|E(u^0_\OD)- \tilde{E}(\tilde{u}^0_\OD)| + \tfrac{\beta}{2} \big|\,\int_\D |u^0_\OD|^4 - P_\OD(|\tilde u^0_\OD|^2)|\tilde u^0_\OD|^2 \dx\,\big|. 
	\end{equation*}
	The difference of the energies can be estimated similarly to~\eqref{proof:crude_est_2} by $C H^4$. For the second difference, we observe that
	\begin{align*}
		&\big|\int_\D |u^0_\OD|^4 - P_\OD(|\tilde u^0_\OD|^2)|\tilde u^0_\OD|^2 \dx\,\big|\\
		=\,& \big|\int_\D (|u^0_\OD|^2-|\tilde u^0_\OD|^2)|u^0_\OD|^2 + (|u^0_\OD|^2-P_\OD| u^0_\OD|^2)|\tilde u^0_\OD|^2
		+ P_\OD(|u^0_\OD|^2-|\tilde u^0_\OD|^2)|\tilde u^0_\OD|^2 \dx\,\big|\\
		\leq\, &\big\|\,|u^0_\OD|^2-|\tilde u^0_\OD|^2\,\big\|\,\big(\|u^0_\OD\|^2_{L^4}+\|\tilde u^0_\OD\|^2_{L^4}\big) + \big\|\,|u^0_\OD|^2-P_\OD(| u^0_\OD|^2)\big\|\,\|\tilde u^0_\OD\|^2_{L^4}\\
		\leq\, &\|u^0_\OD-\tilde u^0_\OD\|_{L^4}\,\big(\|u^0_\OD\|^2_{L^4}+\|\tilde u^0_\OD\|^2_{L^4}\big)^2 + \big\|\,|u^0_\OD|^2-P_\OD(| u^0_\OD|^2)\big\|\,\|\tilde u^0_\OD\|^2_{L^4}\\
		\lesssim\, &H^2  + H^4
	\end{align*}
	holds by $H^1(\D) \hookrightarrow L^4(\D)$, Lemma~\ref{lem:density_replacement}, and~\ref{lem:crude_est}. This finish the proof.
\end{proof}

We are now ready to prove the optimal rate of convergence of $\tilde{u}^0_\OD$; cf.~Theorem~\ref{th:est_VOD_classic}.

\begin{theorem}[Optimal convergence rate for $\tilde u^0_\OD$]\label{th:opt-conv}
	Given the setting of Lemma~\ref{lem:crude_est} and~\ref{lem:est_EW}, 
	let $H$ be small enough, then the energy, the $H^1$-norm, and the eigenvalue of the modified problem converge with optimal order, namely,
	\begin{equation*}
		\|\tilde{u}^0_\OD-u^0\|_{H^1} \lesssim H^3, 
		\qquad
		|E(\tilde{u}^0_\OD)-E(u^0)|\lesssim H^6,
		\qquad
		|\tilde \lambda_\OD^0 - \lambda^0| \lesssim H^3.
	\end{equation*}
\end{theorem}

\begin{proof}
	The main idea behind the proof is to show that 
	\begin{equation*}
		\eta = \tilde{u}^0_\OD-u^0_\OD	\in \VOD
	\end{equation*}
	is of $\mathcal{O}(H^3)$ with respect to the $H^1$-norm (or in the equivalent norm induced by~$a$). 
	We note that, by the convergence of $u^0_\OD$ and $\tilde u^0_\OD$ to the $L^2$-normalised $u^0$, cf.\ Theorem~\ref{th:est_VOD_classic} and Lemma~\ref{lem:crude_est}, $(\tilde{u}^0_\OD,u^0_\OD)\geq 0$ holds for $H$ small enough.
	
	For the proof, we consider the variation of $E$ and $\tilde{E}$ at $u^0_\OD$ and $\tilde{u}^0_\OD$, in the direction of~$\eta$. Observe that, since both are normalised in $L^2$, we have that 
	\begin{equation*}
		\|\eta\|^2 = \|\tilde{u}_\OD^0\|^2 - 2(\tilde{u}_\OD^0,u_\OD^0) + \|u_\OD^0\|^2= 2\big[1-(\tilde{u}_\OD^0,u_\OD^0)\big].
	\end{equation*}
	This allows us to write the first variations as
	\begin{subequations}
		\label{eq:EW}
		\begin{align}
			\tilde{E}^\prime(\tilde{u}^0_\OD)[\eta] &=  \tilde{\lambda}_\OD^0  (\tilde{u}_\OD^0 , \eta) = \tilde{\lambda}_\OD^0  (1-(\tilde{u}_\OD^0,u_\OD^0)) = \hphantom{-}\tfrac 1 2 \tilde{\lambda}_\OD^0 \|\eta\|^2,\\
			E^\prime(u^0_\OD)[\eta] &=  \lambda_\OD^0  (u_\OD^0 , \eta) = \lambda_\OD^0  ((\tilde{u}_\OD^0,u_\OD^0)-1) = -\tfrac 1 2 \lambda_\OD^0 \|\eta\|^2.
		\end{align}
	\end{subequations}
\noindent	On the other hand, we have by a simple expansion 
	\begin{align*}
		E^\prime(\tilde{u}^0_\OD)[\eta] = E^\prime(u^0_\OD)[\eta] + E^{\prime\prime}(u^0_\OD)[\eta,\eta] + R(u^0_\OD,\eta)
	\end{align*}
	with the remainder
	\begin{equation*}
		R(u^0_\OD,\eta) = 2\beta \int_\D 3 u^0_\OD\eta^3 +\eta^4 \dx.
	\end{equation*}
	By a rearrangement, we have 
	\begin{align*}
		\underbrace{\langle (E^{\prime \prime}(u^0_\OD) - \lambda_\OD^0)\eta,\eta\rangle}_{\eqqcolon D_1} &= E^{\prime \prime}(u^0_\OD)[\eta,\eta] + 2 E^\prime(u^0_\OD)[\eta]\\
		&= E^\prime(\tilde{u}^0_\OD)[\eta] + E^\prime(u^0_\OD)[\eta] -  R(u^0_\OD,\eta)\\
		&= \underbrace{E^\prime(\tilde{u}^0_\OD)[\eta] - \tilde{E}^\prime(\tilde{u}^0_\OD)[\eta]}_{\eqqcolon D_2}  + \underbrace{\tilde{E}^\prime(\tilde{u}^0_\OD)[\eta] +  E^\prime(u^0_\OD)[\eta]}_{\eqqcolon D_3} -  \underbrace{R(u^0_\OD,\eta)}_{\eqqcolon D_4}.
	\end{align*}
	In the following, we estimate the terms $D_i$, $i=1,\ldots,4$. 
	
	For the left-hand side, i.e., term~$D_1$, we note that $\gamma \|\eta\|_{H^1}^2 \leq \langle (E^{\prime \prime}(u^0) - \lambda^0)\eta,\eta\rangle$ holds for the continuous minimiser $u^0$ and a constant $\gamma >0$; see \cite[Lem.~1]{CCM10}. On the other hand, the estimate
	\begin{align*}
		&\hphantom{\leq }\,\big| \langle (E^{\prime \prime}(u^0) - \lambda^0)\eta,\eta\rangle - \langle (E^{\prime \prime}(u^0_\OD) - \lambda_\OD^0)\eta,\eta\rangle\big|\\ 
		&\leq 6\beta \int_\D \big|\,|u^0|^2-|u^0_\OD|^2\,|\eta^2 \dx + | \lambda^0 - \lambda_\OD^0|\|\eta\|^2\\
		&\leq 6\beta \|u_\OD^0 + u^0\|_{L^6} \|u_\OD^0 - u^0\|_{L^2} \|\eta\|_{L^6}^2 + | \lambda^0 - \lambda_\OD^0|\|\eta\|^2\\
		& \lesssim  (1+H^3)H^3 \|\eta\|_{H^1}^2
	\end{align*}
	holds by Theorem~\ref{th:est_VOD_classic}. Therefore, for $H$ small enough, the left-hand side is bounded by
	\begin{equation*}
		D_1 = \langle (E^{\prime \prime}(u^0_\OD) - \lambda_\OD^0)\eta,\eta\rangle \geq \tfrac 23 \gamma \|\eta\|^2_{H^1}.
	\end{equation*}
	The second term~$D_2$ is estimated using Lemma~\ref{lem:density_replacement}, i.e.,
	\begin{align*}
		|D_2| = |E^\prime(\tilde{u}^0_\OD)[\eta] - \tilde{E}^\prime(\tilde{u}^0_\OD)[\eta]| &= 2 \beta |\int_\D (  P_\OD(|\tilde{u}^0_\OD|^2)-|\tilde{u}^0_\OD|^2)(P_\OD(\tilde{u}^0_\OD\eta)-\tilde{u}^0_\OD\eta )\dx|\\
		&\leq 2 \beta \|P_\OD(|\tilde{u}^0_\OD|^2)-|\tilde{u}^0_\OD|^2\|\,\|P_\OD(\tilde{u}^0_\OD\eta)-\tilde{u}^0_\OD\eta\|\\
		&\lesssim H^3 \|\tilde{u}^0_\OD\|^3_{H^2}\|\eta\|_{H^1}^{\vphantom{3}}.
	\end{align*}
	By Lemma~\ref{lem:est_EW} and Eq.~\eqref{eq:EW}, the third term is bounded by
	\begin{align*}
		|D_3| = |\tilde{E}^\prime(\tilde{u}^0_\OD)[\eta] +  E^\prime(u^0_\OD)[\eta]| = \tfrac 1 2 (\lambda^0_\OD - \tilde{\lambda}^0_\OD)\|\eta\|^2 \leq C H^{2}\|\eta\|^2 \leq C H^{2}\|\eta\|_{H^1}^2 \leq \tfrac{1}{3} \gamma \|\eta\|_{H^1}^2.
	\end{align*}
	for $H$ small enough. 
	The last term can be estimated, using Lemma~\ref{lem:crude_est}, by
	\begin{equation*}
		|D_4| = |R(u^0_\OD,\eta)| \lesssim \|u^0_\OD\|\,\|\eta\|_{L^6}^3 + \|\eta\|_{L^4}^4 \lesssim (H^4 + H^6) \|\eta\|_{H^1}.
	\end{equation*}
	Putting the estimates of $D_i$, $i=1,\ldots,4$, together and dividing by $\|\eta\|_{H^1}$, we conclude for small enough $H$, that 
	\begin{equation*}
		\tfrac \gamma 3 \|\eta\|_{H^1} = \frac{D_1 - |D_3|}{\|\eta\|_{H^1}} \leq \frac{|D_2| + |D_4|}{\|\eta\|_{H^1}} \lesssim H^3 + H^4 + H^6.  
	\end{equation*}

\noindent	The assertion on the energy then follows by Theorem~\ref{th:est_VOD_classic}, the triangle inequality, and 
	\begin{align*}
		|E(\tilde{u}^0_\OD) - E(u^0_\OD)| &\leq |E^\prime(u^0_\OD)[\eta]| + \tfrac 1 2 |E^{\prime\prime}(u^0_\OD)[\eta,\eta]| + \frac{\beta}{2} \int_\D 4|u^0_\OD\eta^3| + \eta^4 \dx\\
		& \lesssim \lambda_\OD \|\eta\|^2 + \|\eta\|_a^2 + \|\eta\|_{H^1_0}^2 \sum_{k=0}^2 \|u^0_\OD\|_{H^1_0}^{2-k}\|\eta\|_{H^1_0}^{k} = \mathcal O(H^6).
	\end{align*}
	Finally, with the new estimate of $\|\eta\|_{H^1}$, Theorem~\ref{th:est_VOD_classic} and the steps in Lemma~\ref{lem:est_EW} imply the third order convergence of  $|\tilde{\lambda}^0_\OD-\lambda^0|$.
\end{proof}

\begin{remark}\label{rem:a_with_Vd}
	The statement of Theorem~\ref{th:opt-conv} is still valid, if we split $V$ into $V=V_{\text{s}}+V_{\text{d}}$ with $V_{\text{d}} \in L^{2+\sigma}(\D)$ and $V_{\text{s}} \in H^2(\D)$.
	Here, the bilinear form $a$ is defined only with $V_{\text{d}}$ - instead of the whole potential~$V$. Note that, this implies  
	\begin{equation*}
		-\tfrac 1 2 	\triangle u^0 + V_{\text{d}} u^0 = \lambda^0 u^0 - \beta |u^0|^2u^0 - V_{\text{s}} u^0 \in H^2(\D)\cap H^1_0(\D)
	\end{equation*}
	and therefore $\|u^0_\OD - u^0\|_{H^1(\D)} \lesssim H^3$ by the steps of \cite[Prop.~4.2]{HeP21}. The remaining proof follows then the one with the complete bilinear form $a$, where one has to consider the additional term $\|V_{\text{s}}u_\OD^0\|\lesssim \|V_\text{s}\|_{L^\infty(\D)}$ in Lemma~\ref{lem:H2_bound_u0_OD} and similar in Lemma~\ref{H^2-tilde}, as well as use
	$|E^{\prime\prime}(u^0_\OD)[\eta,\eta]| \lesssim \|\eta\|_a^2 + \|V_{\text{s}}\|_{L^\infty(\D)} \|\eta\|^2_{L^2(\D)}$
	for the bound of $|E(\tilde{u}^0_\OD) - E(u^0_\OD)|$ in the proof of Theorem~\ref{th:opt-conv}.
\end{remark}

\begin{corollary}\label{4th-order-bound}
	Given the setting of Theorem~\ref{th:opt-conv}, the modified minimal energy is guaranteed to converge with order  $\mathcal{O}(H^4)$, i.e., $|\tilde{E}(\tilde{u}^0_\OD)-E(u^0)|\lesssim H^4$.
\end{corollary}

\begin{proof}
	The assertion follows immediately by Lemmas~\ref{lem:density_replacement}, \ref{H^2-tilde}, Theorem~\ref{th:opt-conv}, and
	\begin{equation*}
		|\tilde{E}(\tilde{u}^0_\OD)-E(\tilde{u}^0_\OD)| = \int P_\OD(|\tilde{u}^0_\OD|^2) - |\tilde{u}^0_\OD |^2 )|\tilde{u}^0_\OD |^2d dx = \int \big(P_\OD(|\tilde{u}^0_\OD|^2) - |\tilde{u}^0_\OD |^2 \big)^2 dx\lesssim H^4. \qedhere
	\end{equation*}
\end{proof}

\begin{remark} If the potential and the domain are sufficiently smooth, the modified minimum energy is guaranteed to converge with the optimal rate of order $\mathcal{O}(H^6)$, furthermore $|\tilde{E}(\tilde{u}^0_\OD)-E(\tilde u^0_\OD)|\lesssim H^8.$ This can be proved following the lines of Lemma~ \ref{lem:density_replacement}, where with sufficient smoothness Eq.~\eqref{ProjectionError2} can be replaced by the optimal \eqref{OD-L2-est}.
\end{remark}

\begin{remark}
	By the convergence of $\tilde{u}^0_\OD$ and by the inequality $\tilde{E}(u^0)\leq E(u^0)$, see  \eqref{eq:tildeE<=E}, we have the ordering
	\begin{equation*}
		\tilde{E}(\tilde u^0_\OD)\leq E(u^0) \leq E(\tilde u^0_\OD)
	\end{equation*}	
	for small enough $H$. In particular, $E(\tilde u^0_\OD) - \tilde{E}(\tilde u^0_\OD)$ is a simple estimator 
	for the approximation error of the energy and therefore also of $\|\tilde{u}^0_\OD-u^0\|_a^2$.
\end{remark}

\subsection{Error estimates for SLOD spaces} \label{truncation-error}
In the previous subsection we analysed the $u_\OD^0-u^0$ and $\tilde u_\OD^0-u^0$ errors in the ideal OD space. While it is not known whether a localised basis exists for this space, the way we currently compute the basis leads to a truncation of the global, but rapidly decaying, support of the basis functions. One may well ask how this truncation affects the error bounds. In this section, we provide a qualitative answer to this question. Note that the discrete function space has three discretisation parameters: the patch size~$\ell$ (truncation parameter), the mesh width~$H$ of the piecewise linear right-hand sides, and the width~$h$ of the mesh used to discretise the responses. In the following, we make the simplification that $h=0$ holds. For $h>0$ the steps are equivalent, one only has to use the triangle inequality to introduce the representation error.    

The SLOD space with $h=0$ is denoted ~$\VSLOD$. Let~$p^*_z$ be the minimiser of the localisation error problem~\eqref{eq:choice_SLOD} with respect to the node~$z$. The associated SLOD function is $\varphi_{\SLOD,z}$ and $\varphi_{\OD,z}=\mathcal{A}^{-1}p^*_z$ is the associated ideal OD function. Furthermore, we make the typical assumption for a SLOD discretisation that $\{p^*_z\}_{z\in \mathcal{N}_H}$ forms a Riesz basis of~$\mathbb{P}_H^1$, i.e.~there exists a $C_{\text{RB}} > 0$ depending on $H, \ell$ such that
\begin{equation*}
	C^{-1}_{\text{RB}} \sum_{z \in \mathcal{N}_H} c_z^2 \leq \Big\| \sum_{z\in \mathcal{N}_H} c_z p^\ast_z \Big\|^2_{L^2(\D)} \leq C_{\text{RB}} \sum_{z\in \mathcal{N}_H} c_z^2 
\end{equation*}
holds for every $\{c_z\}_{z \in \mathcal{N}_H}$, cf.~\cite[Ass.~5.2]{HaPe21}. By this assumption, the ideal OD minimiser~$u^0_\OD$ of $E$ can be written as $\sum_{z\in \mathcal{N}} c_z^0 \varphi_{\OD,z}$.  We define $v^0_\SLOD \coloneqq \sum_{z\in \mathcal{N}} c_z^0 \varphi_{\SLOD,z} \in \VSLOD$. Furthermore, the minimiser of $E$ in SLOD space is denoted by $u^0_{\SLOD}$. Then, by the convergence of $u^0_\OD$ and \cite[Lem.~1 \& Eq.~(32)]{CCM10}, we have the estimate 
\begin{equation}\label{eqn:est_E_SLOD}
	\begin{aligned}
	E(u^0_\SLOD) - E(u^0)
	\leq\, & E(v^0_\SLOD)- E(u^0)\\
	\lesssim\, &\|u^0 - v^0_\SLOD\|_{a}^2 + \int_\D \big((u^0)^2 - (v^0_\SLOD)^2\big)^2 \dx\\
	\lesssim\, & \big(1 + \|u^0 + v^0_\SLOD \|_{a}^2\big) \|u^0 - v^0_\SLOD \|_{a}^2\\
	\lesssim\, & \big(1 + \|u^0_\OD -  v^0_\SLOD\|_{a}^2\big) \big( \|u^0 - u^0_\OD \|_{a}^2  +  \|u^0_\OD - v^0_\SLOD\|_{a}^2  \big)\\
	\lesssim\, &\big(1+H^{-1}\sigma^2(H,\ell)\big)\big(H^6 + H^{-1}\sigma^2(H,\ell)\big).
\end{aligned}
\end{equation}
Here, $\sigma(H,\ell)$ is the maximum of the localisation errors over all nodes $z\in \mathcal{N}_H$. In particular,  we used  
\begin{equation*}
	\| u^0_\OD - v^0_\SLOD\|_{a} \lesssim \bigg(\frac {C_{\text{RB}}(\ell+1)^{d-1}\big(1 + \sqrt d (\ell +1)H\big)}{H}\bigg)^{\nicefrac{1}{2}} \,\sigma(H,\ell)\, \|u^0_\OD\|_{H^2(\D)}  
\end{equation*}
in the last step, which follows by
\begin{equation*}
	\|e\|^2_{L^2(\partial B_r(z))} \leq \frac {1 +\sqrt{1+dr^2}}{r} \|e\|^2_{H^1(B_r(z))} \leq \Big( \frac{2}{r}+\sqrt{d}\Big) \|e\|^2_{H^1(B_r(z))},
\end{equation*}
cf.~\cite[p.~41]{Gri85}, and by the lines of \cite[Th.~6.1]{HaPe21}.  

As for the effect of truncation on the modified minimiser, we expect it to be proportional to $\sigma$ as measured in the $\|\bullet\|_a$ norm.  However, we can only support this statement heuristically. The projection $P_\OD$ in the definition of~$\tilde E$ must be replaced by the $L^2$-projection~$P_{\SLOD}$ onto~$\VSLOD$. Let us denote the minimiser of~$\tilde{E}$ in~$\VSLOD$ by ~$\tilde u^0_\SLOD$. Analogous to~$v^0_{\SLOD}$, we define~$\tilde v^0_{\SLOD}$ by the coefficients of~$\tilde u^0_{\OD}$. Under the assumption that $E(\tilde u^0_\SLOD) \leq E(\tilde v^0_\SLOD)$, we have $E(\tilde u^0_\SLOD) - E(u^0) \leq  E(\tilde v^0_\SLOD) - E(u^0)$, which can be estimated similarly to $E(v^0_\SLOD) - E(u^0)$ in~\eqref{eqn:est_E_SLOD}, thus supporting the claim. However, if 
$E(\tilde u^0_\SLOD) \leq E(\tilde v^0_\SLOD)$ can not be asserted then an extra term according to 
\begin{align*}
&E(\tilde u^0_\SLOD) - E(u^0)\\
\leq\, &E(\tilde u^0_\SLOD) - \tilde{E}_\ell(\tilde u^0_\SLOD) + \tilde{E}_\ell(\tilde v^0_\SLOD) - E(\tilde v^0_\SLOD) + E(\tilde v^0_\SLOD)- E(u^0)\\
=\, &\tfrac \beta 2 \big\|\,|\tilde u^0_\SLOD|^2 - P_{\SLOD} |\tilde u^0_\SLOD|^2\big\|_{L^2(\D)}^2 - \tfrac \beta 2 \big\|\,|\tilde v^0_\SLOD|^2 - P_{\SLOD} |\tilde v^0_\SLOD|^2\big\|_{L^2(\D)}^2 + E(\tilde v^0_\SLOD) - E(u^0), 
\end{align*}
would have to be estimated.
\subsection{Minimisation algorithm} 
It remains to choose an appropriate algorithm for computing the minimiser. There have been numerous proposals for minimisation algorithms to compute the ground states of BECs: From early work on energy-diminishing backward Euler centered finite difference (BEFD) methods and explicit time-splitting sine-spectral (TSSP) methods \cite{AftalionDu, BaD04} to the energy-adapted Riemannian gradient descent method with global convergence property for non-rotating BECs \cite{SobolevGradient, 2021arXiv210809831A}, inverse iteration methods \cite{Elias, altmann2021j, LandscapeBEC} with local quadratic convergence, Krylov preconditioned BESP \cite{KrylovPrecond}, Preconditioned Gradient Descent \cite{PCG} for rapidly rotating BECs, and Riemannian Newton methods \cite{2023arXiv230713820A}. To compute the minimiser, we consider the combined method of energy-adapted Riemannian gradient descent \cite{SobolevGradient,2023arXiv230713820A} and inverse iteration, called the J-method, first introduced in the finite difference setting \cite{Elias}, then extended to and analysed in the finite element setting \cite{altmann2021j}. The J-method allows both selective approximation of excited states and cubic convergence in a local neighborhood of a minimum, similar to inverse iteration for linear eigenvalue problems. In addition, the gradient descent approach is energy-decreasing and is guaranteed to converge to the unique global minimiser for non-rotational cases.

\section{Numerical experiments for stationary states}\label{sec:Numerics}
In this section we present several numerical examples to illustrate the previous results, as well as comparisons with GPELab \cite{gpelab}, the classical LOD approach \cite{HeWaDo22}, and BEC2HPC  \cite{BEC2HPC}. 
All CPU times were measured on a laptop with an 11th generation Intel${}^{\text{\textregistered}}$ Core™ i7-1165G7
@ 2.80GHz × 8 processor and 64GB of RAM running Julia version 1.7.2 (2022-02-06). Note that the current implementation is strictly sequential. The code is available at \href{https://github.com/JWAER/SLOD\_BEC}{https://github.com/JWAER/SLOD\_BEC}.
\subsection{Smooth academic example in $2$d} \label{sec:smooth2d}
We begin with a smooth test case from \cite{HeWaDo22} which will allow us to compare our SLOD with the previous LOD implementation. The problem reads,
\begin{align*}
	\min_{u\in \mathbb{S}}  \int_\D \frac{1}{2}|\nabla u|^2 + V|u|^2 + \frac{\beta}{2}|u|^4 \dx,
\end{align*}
with 
\begin{align*}
	V(x,y) &= \frac{1}{2}(x^2+y^2)+4e^{-x^2/2}+4e^{-y^2/2},\quad  \beta= 50,\quad \text{and}\quad \D= (-6,6)^2.
\end{align*}
Since the potential is smooth, the canonical OD-space is used, i.e., the basis functions span the space $	\triangle^{-1}\mathbb{P}^1_H(\mathcal{T}_H)$.
In Fig.~\ref{fig:smooth$2$d_conv} is plotted the convergence rate versus mesh size for the SLOD with $h=H$ and $h=H/2$, the LOD results from \cite{HeWaDo22}, and the GPELab, for which an ersatz mesh size is computed as one over the number of modes in each direction. The minimal energy is computed to 10 digit accuracy to be $E^0=7.082310561$ which differs slightly from the minimal energy of the spectral solution with periodic boundary conditions, $E^0=7.082310558$. As is illustrated in Fig.~\ref{fig:smooth$2$d_conv} no order of accuracy is initially lost by representing the SLOD basis by $\mathbb{P}^3_H$-functions on the same mesh.
However, the final data point shows that as the solution near the boundary starts to come into play, a finer representation of the basis functions becomes necessary. Strikingly the order of convergence is 7, as opposed to the estimated $6$. We find that setting the truncation parameter $\ell=2$ is sufficient to push the error to 10 digit accuracy. Although not illustrated, it should be mentioned that the choice $\ell = 1$ is good enough to achieve $5$-digit accuracy, after which the error stagnates unless $\ell$ is increased. This is in accordance with the estimate in Section \ref{truncation-error} and the computational results in Fig.~\ref{fig:SLOD_decay}. It is clear from Fig.~\ref{fig:smooth$2$d_CPU}, in which accuracy versus CPU-time is plotted, that the improvement from the LOD-method is several orders of magnitude. Moreover, even for this analytical test case, the SLOD rivals the spectral method as both methods almost instantaneously yield 8-digit accuracy. Finally we report that our minimisation algorithm converged in 8--9 iterations when switching to the J-method at a residual lower than $0.1$, the corresponding number of iterations for the pure Sobolev gradient descent was reported in \cite{HeWaDo22} to be around 20.

	\begin{figure}	
	\begin{subfigure}{0.49\textwidth}
		\centering
		\includegraphics[width=1.1\linewidth]{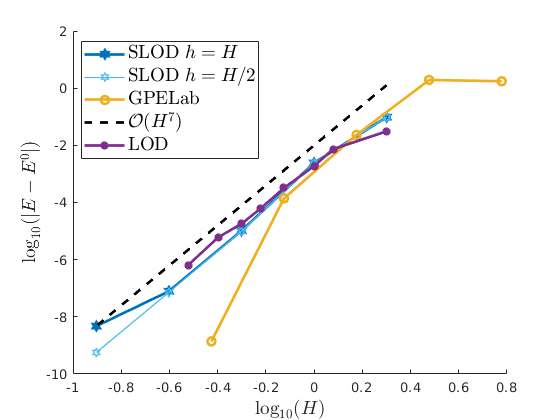}
		\caption{Convergence w.r.t.\ $H$, for the spectral method an ersatz mesh size is computed as one over the number of modes in each direction.}
		\label{fig:smooth$2$d_conv}
	\end{subfigure}\quad
	\begin{subfigure}{0.49\textwidth}
		\centering
		\includegraphics[width=1.1\linewidth]{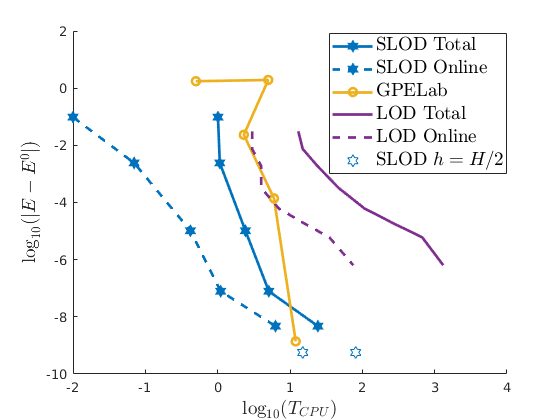}
		\caption{Accuracy versus CPU time in secounds for the three methods.\\~}
		\label{fig:smooth$2$d_CPU}
	\end{subfigure}
	\caption{Error in minimal energy versus mesh size and error in minimal energy versus CPU time for the SLOD, LOD and GPELab when solving the problem in section \ref{sec:smooth2d}.}
\end{figure}

\subsection{Discontinuous potential in $2$d}\label{sec:discont1}
Next we consider a potential with multiple discontinuities. Naturally in such cases, the difference between the SLOD and a spectral method becomes very pronounced. Not only does this experiment exemplarily illustrate this, but it poses a difficult challenge for other methods to solve to the same accuracy at similar CPU-time costs. The problem reads,
\begin{align*}
 \min_{u\in\mathbb{S}} \int_\D \frac{1}{2}|\nabla u|^2 + (V_{\text{s}}+V_{\text{d}})|u|^2 + \frac{\beta}{2}|u|^4 \dx,
\end{align*}
there now being an additional discontinuous contribution to the potential, namely
\begin{align*}
V_{\text{s}}(x,y) = \frac{1}{2}(x^2+y^2), \quad V_{\text{d}}(x,y)  = \Big\lfloor 5+2\sin\Big(\frac{\pi x}{3}\Big)\sin\Big(\frac{\pi y}{3}\Big) \Big\rfloor, \quad \text{and} \quad \beta    = 100,
\end{align*}
where $\lfloor \bullet  \rfloor$ denotes integer part. The inner-product used to define the OD-space is chosen in accordance with Remark~\eqref{rem:a_with_Vd} to be \begin{align*}
	a_{\text{d}}(u,v) = \int_\D \frac{1}{2}\nabla u\cdot\nabla v + V_{\text{d}} uv \dx.
\end{align*}
 For this example, we consistently use a refined mesh of size $h = H/3$ with $\mathbb{P}^3_h$ functions to represent the SLOD-basis functions and truncation parameter $\ell = 2$. The discontinuous potential is integrated to machine precision using adaptive quadrature around jumps. A very fine reference solution is used to compute the minimal energy to be, to $10$-digit precision, $8.30472428538$. 

\begin{figure}	
	\begin{subfigure}{0.49\textwidth}
		\centering
		\includegraphics[width=1.1\linewidth]{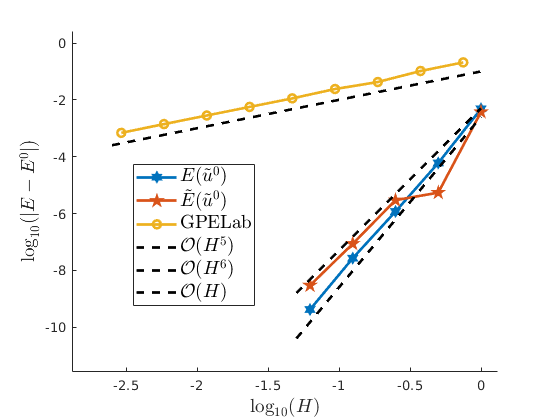}
		\caption{Convergence w.r.t.\ $H$, for the spectral method an ersatz mesh size is computed as one over the number of modes in each direction.}
		\label{fig:rough$2$d_conv}
	\end{subfigure}\quad
	\begin{subfigure}{0.49\textwidth}
		\centering
		\includegraphics[width=1.1\linewidth]{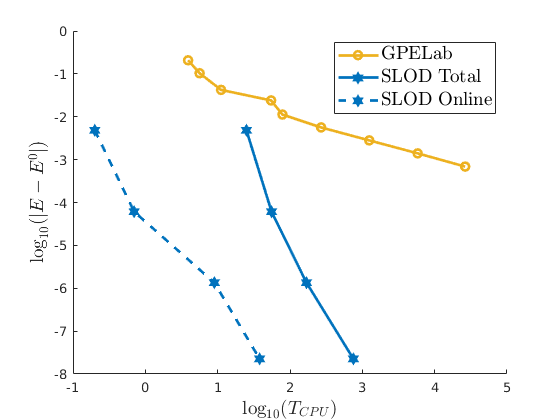}
		\caption{Accuracy versus CPU time in seconds for the proposed method and the spectral GPElab.\\~}
		\label{fig:roudh$2$d_efficiency}
	\end{subfigure}
	\caption{Accuracy versus mesh size and accuracy versus CPU times for the SLOD and GPELab in presence of the discontinuous potential of section \ref{sec:discont1} }
\end{figure}
 As illustrated in Fig.~\ref{fig:rough$2$d_conv}, in which the error in energy versus mesh size is plotted, the spectral method converges with linear order due to the Fourier modes of the discontinuous potential only decaying linearly whereas the SLOD-method does indeed converge with perfect 6th order as predicted by Theorem \eqref{th:opt-conv}. We also note how $\tilde{E}$, i.e., the strategy in \cite{HeWaDo22}, converges with $\mathcal{O}(H^5)$. Similarly to the previous example, around 10 iterations were required when switching to the inverse iteration at a residual of $0.1$. In sum, we are able to compute the minimal energy to 8-digit accuracy in less than 1000s total computational time. 

\subsection{A challenging example with fourth order convergence rate in $\tilde{E}$}\label{sec:discont2}
This additional experiment illustrates that the convergence of $\mathcal{O}(H^4)$ of $\tilde{E}$ in Corollary~\ref{4th-order-bound} is sharp. We change the discontinuous potential of the previous subsection to
\begin{align*}
 \quad V_{\text{d}}(x,y)  = 2(\mathbbm{1}_{x>0}+\mathbbm{1}_{y>0}),
\end{align*}
using the indicator function $\mathbbm{1}_{(\cdot)}$ 
but keep all other parameters equal. Again, the inner-product \begin{align*}
a_{\text{d}}(u,v) = \int_\D \frac{1}{2}\nabla u\cdot\nabla v + V_{\text{d}} uv \dx, 
\end{align*}
is used. For this example we set $h=H/2$, but observe that $h=H$ also yielded optimal convergence rates until the last two data points (provided the mesh matched the discontinuities). In Fig.~\ref{fig:rough$2$d_2_conv} we see that $E(\tilde{u}^0_\OD)$ converges slightly faster than the $\mathcal{O}(H^6)$ but the convergence of $\tilde{E}(\tilde{u}^0_\OD)$ is closer to 4. For the finest mesh size, the difference in accuracy is more than one order of magnitude. Since the potential is easy to integrate and $h=H/2$ suffices, the CPU times for this example are roughly a tenth of the ones in the previous example, cf. Fig.~\ref{fig:roudh$2$d_2_efficiency}.
\begin{figure}
	\begin{subfigure}{0.49\textwidth}
		\centering
		\includegraphics[width=1.1\linewidth]{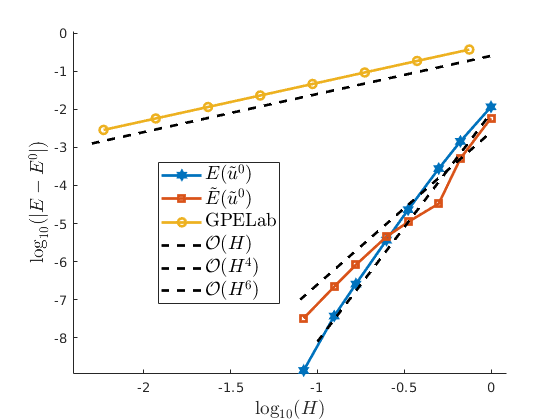}
		\caption{Convergence w.r.t.\ $H$, for the spectral method an ersatz mesh size is computed as one over the number of modes in each direction.}
		\label{fig:rough$2$d_2_conv}
	\end{subfigure}\quad
	\begin{subfigure}{0.49\textwidth}
		\centering
		\includegraphics[width=1.1\linewidth]{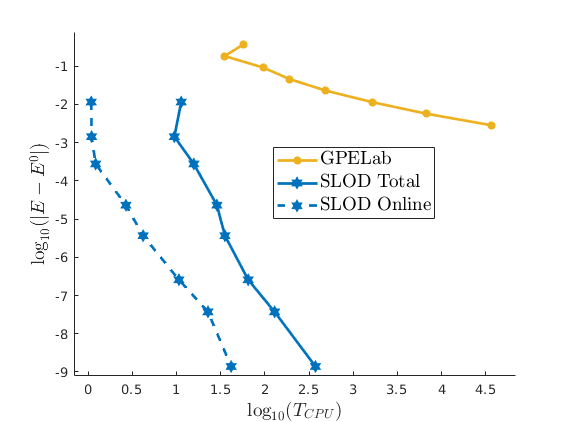}
		\caption{Accuracy versus CPU time in seconds for the proposed method and the spectral GPElab.\\~}
		\label{fig:roudh$2$d_2_efficiency}
	\end{subfigure}
	\caption{Accuracy versus mesh size and accuracy versus CPU times for the SLOD and GPELab in presence of the discontinuous potential of section \ref{sec:discont2}.}
\end{figure}

\subsection{Fast rotation in $2$d} \label{sec:FastRotation}
We consider the challenging rotational experiment put forward in \cite{altmann2021j}, which reads
\begin{align} \label{MinProblem2}
&\min_{ v \in\mathbb{S} } \int_\D |\nabla v|^2+V|u|^2+v\Omega \overline{ L_zv} +\frac{\beta}{2}|u|^4 \dx,
\end{align}
where $\beta = 1000$, $\Omega = 0.85$, $V(x,y) = \tfrac{1}{2}(x^2+y^2)$ and $\D = (-10,10)^2$. When the residual of the nonlinear eigenvalue problem is less than $3\cdot 10^{-3}$, we switch from the Sobolev gradient descent to the inverse iteration of the J-method. We recall that the Sobolev gradient method, which in each iteration performs a line search to find the optimal step-size, has a tendency to get stuck at local minima. To circumvent this issue, the algorithm is initialized pointwise on the fine grid as  $ ( (1-\Omega)\xi_1+ \Omega(x+\ci y)\xi_2 )e^{-(x^2+y^2)}$, where  $\xi_1$ and $ \xi_2$ are, for each grid point, independent random numbers uniformly distributed between $[0,1]$. This is then projected onto $\mathcal{V}_\OD$ in the $L^2$-sense. As a result the necessary number of iterations can vary greatly, cf.\ Table~\ref{Table_Rot}. On some occasions this approach converged to an excited state close to the presumptive ground state. These excited states are illustrated in Fig.~\ref{fig:Excited1} through~\ref{fig:Excited3} and  were not found in \cite{altmann2021j}. It is interesting to note that the third excited state  has the lowest eigenvalue and possesses the most symmetry as it is, pointwise, invariant under rotation of $2\pi/8$ radians. Note that we do not check whether the stationary points truly are local minima, as opposed to, e.g., saddle points. This could be done using the constrained high-index saddle dynamics method as recently proposed in \cite{LandscapeBEC}. Curiously, we find higher than expected order of convergence w.r.t.\ the mesh size~$H$, namely~$\mathcal{O}(H^7)$, cf.\ Table~\ref{Table_Rot}, instead of the expected~$\mathcal{O}(H^6)$. The two inbuilt initial values of GPELab, Gaussian and Thomas-Fermi, converged to the energy levels $12.03756641$ and $10.76768160$ respectively using the  Backward Euler sine pseudospectral (BESP) method with step-size 1e-2. The improved BEC2HPC, high performance spectral solver for rapidly rotating large BECs converged after 2737 iterations in 5 min to $E=10.727491588$ and $\lambda = 15.56205302 $ on an 128x128 grid.  On a 256x256 grid it converged in 905 iterations and 10min to $ E = 10.71847990$ and $\lambda = 15.60414509$. As can be read off of Table \ref{Table_Rot}, our method converged to the presumptive ground state already for the coarse 60x60 grid in a few minutes. In around 5min of computation, the method computed the ground state energy to 5 digit accuracy on an 80x80 grid. However, the accuracy in terms of energy of the 256x256 spectral solution is approximately $10^{-8}$, to achieve such accuracy with our method required around 1.5h.
\begin{figure}\label{fig:rot}	
	\captionsetup[subfigure]{justification=centering}
	\begin{subfigure}{0.49\textwidth}
		\centering
		\includegraphics[width=0.8\linewidth]{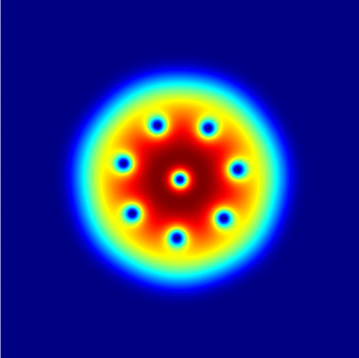}
		\caption{$E = 10.718480= 2\cdot 5.359240$,\\ $\lambda = 15.604146, \ H = 20/160 $}
		\label{fig:ground_state}
	\end{subfigure}
	\begin{subfigure}{0.49\textwidth}
	\centering
	\includegraphics[width=0.8\linewidth]{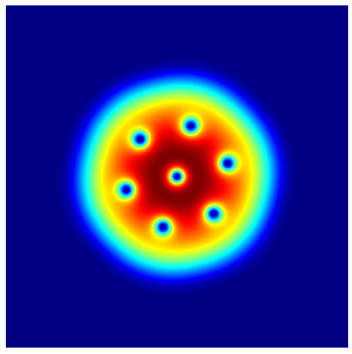}
	\caption{$E = 10.725428 = 2\cdot 5.362714 $\\
		$\lambda = 15.694028, \ H = 20/140 $}
	\label{fig:Excited1}
\end{subfigure}
 \\
\begin{subfigure}{0.49\textwidth}
	\centering
	\includegraphics[width=0.8\linewidth]{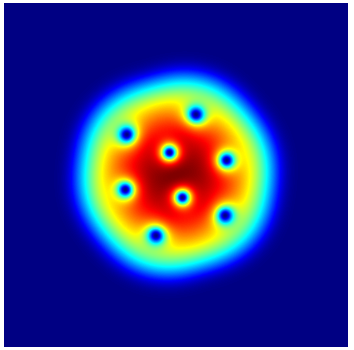}
	\caption{$E = 10.730394 = 2\cdot5.365197  $\\$\lambda= 15.624624, \ H = 20/160 $}
	\label{fig:Excited2}
\end{subfigure}
\begin{subfigure}{0.49\textwidth}
	\centering
	\includegraphics[width=0.8\linewidth]{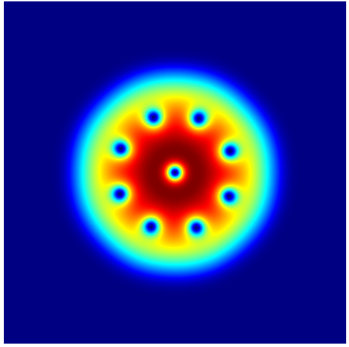}
	\caption{$E = 10.732307 = 2\cdot 5.366154  $\\$\lambda= 15.548887, H = 20/160$}
	\label{fig:Excited3}
\end{subfigure}
\caption{Ground state and excited states of the minimisation problem Eq.~\eqref{MinProblem2}}
\end{figure}
\begin{table}[H]
\caption{Minimal energy and corresponding eigenvalue for different mesh sizes as well as the number of iterations and total CPU-times. Using Richardson extrapolation on the SLOD results, we compute a reference value $E^0=10.71847995$.  }
\label{Table_Rot}
\centering
\begin{tabular}{|l|c|c|c|c|c|}
	\hline
	$\ \ \ H$ & $E$ &EOC &$\lambda$ &N$^\text{o}$ it & T$_\text{ CPU}$	\\
\hline	
$20/60$&10.7191233 & 7.66 & 15.602838 & 369 & 147 \\
$20/80$&10.7185510 & 7.50 & 15.604177 & 267 & 218\\
$20/100$&10.7184932 & 7.23 &15.604165 & 912 & 1692 \\
$20/120$&10.7184835 & 7.02 &15.604152 & 365  & 923\\
$20/140$&10.7184811 & 7.00 &15.604148 & 690 & 2381\\
$20/160$&10.7184804 & & 15.604146 & 992 & 4464 \\ \hline
\end{tabular}
\end{table}
\subsection{Harmonic potentials in $2$d and $3$d} \label{harmonic}
We recall that in the previous work on the method of LOD for solving the GPE \cite{HeWaDo22}, $\mathcal{O}(H^9)$ convergence was observed in $2$d when computing the non-rotating ground state subject to a harmonic potential. In the light of the extra convergence observed in the previous example we consider the problem of minimising the energy in presence of a purely harmonic potential, i.e., 
\begin{align*}
\min_{v\in \mathcal{S}} \int \frac{1}{2} |\nabla v|^2+V|v|^2+\frac{\beta}{2}|v|^4 \dx,
\end{align*}
with $\beta = 50$, $V(x,y) = \tfrac{1}{2}(x^2+y^2)$. The computational domain is set to $(-10,10)^2$ and  $(-5,5)^3$ in $2$d and $3$d, respectively. As described in \cite{HeWaDo22}, the stationary problem can be reduced to a $1$d problem, which when solved to 14 digits accuracy yields  $E^0 = 2.896031852200792$  in $2$d and $E^0 = 2.3734292669786$ in $3$d\cite{HeWaDo22}. Just as in the previous example our method converges with 7th order in both $2$d and $3$d, see Fig.~\ref{fig:harmonic}. This is lower than the $9$th and 12th order observed in \cite{HeWaDo22}. The difference in convergence can be attributed to the difference in coarse $\mathbb{P}^1_H$-spaces. More precisely, the mesh used in this paper consists of isosceles triangles and the type of tetrahedron illustrated in Fig. \ref{fig:mesh_$3$d} whereas \cite{HeWaDo22} used equilateral and nearly regular tetrahedra.
\begin{figure}[H]
	\centering
	\includegraphics[width=0.7\linewidth]{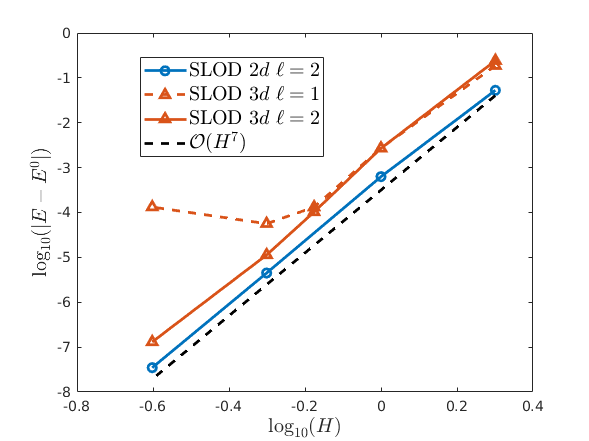}
	\caption{Accuracy of minimal energy versus mesh size in $2$d and $3$d in the case of a purely harmonic potential.}
	\label{fig:harmonic}
\end{figure}

\section{Simulation of the dynamics}\label{sec:temp}
This section aims to demonstrate the usefulness of the SLOD space beyond ground state calculations. Its extreme efficiency will be demonstrated in a combined ground state and dynamics problem in a physically relevant parameter regime in $3$d. The code is available at \href{https://github.com/JWAER/SLOD\_BEC}{https://github.com/JWAER/SLOD\_BEC}.
\subsection{Temporal discretisation}
To integrate in time, we use the continuous Galerkin in time approach introduced by Karakashian and Makridakis \cite{KaM99} and further analysed with the addition of potential terms and rotation in \cite{doeding2022uniform}. The method is energy conservative and superconvergent at the time nodes, where the order of convergence is $\mathcal{O}(\tau^{2q})$ with the polynomial degree $q$ in time and the time step size $\tau$. The benefit of energy conservation in the discrete setting by choosing an appropriate time integrator has been observed in for example \cite{NLSComparison, Superconv}.  Apart from the arbitrarily high order of convergence, the method remains energy conservative, albeit in a modified sense, when replacing $|u|^2$ by its projection onto the adapted operator, thus making it very fast in our setting. Moreover, the modified sense in which it is energy conservative has been rigorously proved to be only $\mathcal{O}(H^8)$ away from the true energy in certain situations \cite{Superconv}. The exact implementation is described exemplarily in \cite{HeWaDo22}, so we will limit this section to a self-contained and high-level outline of the ideas and properties of the approach. In an abstract seeting, the approach seeks the best approximation in the space
\begin{align*}
\{v \in C([0,T], H^1_0(\mathbb{C},\D)):v|_{t\in I_n} \in \VOD \otimes \mathbb{P}_{q}(I_n) \},   
\end{align*}
where $\mathbb{P}_q(I_n)$ denotes the space of polynomials of order $q$ on the time slab $I_n = (t_n,t_{n+1}]$.  The solution can then be recursively defined on each time slab through
\begin{align} \label{fully-discrete}
	\int_{I_n} \langle \ci \partial_t u^n_\OD, v\rangle -a(u^n_\OD,v)-\beta \langle P_\OD(|u^n_\OD|^2)u^n_\OD,v \rangle dt = 0& \quad \forall\, v \in \VOD\otimes \mathbb{P}_{q-1}(I_n), \\
	\lim_{t \downarrow t_n} u^n_\OD(t) = u^{n-1}_\OD(t_n)&,\nonumber
\end{align}
with $u^0_\OD(0)= P_\OD(u^0)$. We point out the consistent replacement of $|u^n_\OD|^2$ with $P_\OD(|u^n_\OD|^2)$. Since $\partial_t u^n_\OD \in \VOD \otimes \mathbb{P}_{q-1}(I_n)$, we may select $v=\partial_t u^n_\OD$ to deduce energy-conservation in the sense that,
\begin{align}\label{ConservedEnergy}
\tilde E(u_{\OD}^n) := \int_\D a(u^n_\OD,a^n_\OD) + \frac{\beta}{2}P_\OD(|u^n_\OD|^2)|u^n_\OD|^2 dx = \tilde E(u^0_\OD).
\end{align}
Although Eq.~\eqref{fully-discrete} is posed in a space of dimension $(q+1)\dim(\VOD)$, the linear part of the system of equations can be decoupled for each time slice, so that only $q$ systems of equations of size $\dim(\VOD)$ need to be solved \cite{KaM99}. A fixed-point iteration is then used to solve the nonlinear system at each time step. 

\subsection{Optical lattice in $3$d}
Inspired by the famous experiment of Greiner et al.\ \cite{Nature}, we study the dynamics of a BEC released from an optical lattice. The initial state is computed as an energy minimiser subject to an optical lattice and a harmonic trapping potential. For $t>0$ the optical lattice is switched off, but the trapping potential is slightly increased to limit the computational domain. The exact parameters are given as follows:
\begin{align*}
E_\text{\tiny gs}= \min_{ v \in\mathbb{S}} \int_\mathcal{D}\frac{1}{2}|\nabla v|^2+(V_s+V_o)v^2+\frac{\beta}{2}|v|^4 dx,
\end{align*}
\begin{align*}
V_s(x,y,z) = x^2+y^2+z^2, \ V_o(x,y,z) = 100\sin^2(\pi x)\sin^2(\pi y)\sin^2(\pi z),  \ \beta = 100,\, \D = (-6,6)^3.
\end{align*}
For the dynamics we set $V(x,y,z) = 2(x^2+y^2+z^2)$ but keep everything else the same and compute to $T=1$ with a time step size of $\tau = 1/128$ and a fourth order method, i.e.~$q=2$. The grid size is $H = 0.25$.
The minimum energy is calculated to be $E_\text{\tiny gs}(\tilde u_\OD^0)=74.585793$ (the modified energy is $\tilde E_\text{\tiny gs}(\tilde u^0_\OD) = 74.488309$).
At $t=0$, the condensate is in a fine lattice structure that decays faster than exponentially away from the origin, cf.\ Fig.~\ref{fig:u0} with the conserved energy of eq. \eqref{ConservedEnergy} being $\tilde E(u_\OD(t) ) = 39.848597$  (due to removal of the optical lattice). This pattern quickly dissipates and at $t=0.4$ no clear structure is visible, cf.\ Fig.~\ref{fig:u1} (note the different scales). At $t=0.8$ a macroscopic lattice structure appears, cf.\ Fig.~\ref{fig:u2}, which subsequently hits the confining potential wall and starts to collapse at $t=1.0$, see Fig.~\ref{fig:u4}.
The dynamics match those of the physical experiment \cite{Nature}. A well-known heuristic explanation of the observed macroscopic pattern is that it is approximately that of the Fourier transform of the Thomas-Fermi approximation of the initial value. In other words, at some later time, the momentum distribution of the initial value is observed. However, the exact influence of the nonlinearity remains an open question.  
\begin{figure}[H]	
	\begin{subfigure}{0.49\textwidth}
		\centering
		\includegraphics[width=0.8\linewidth]{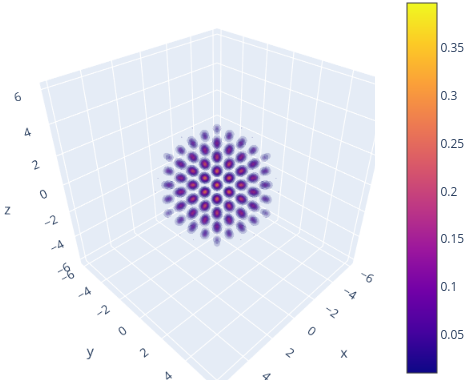}
		\caption{$t=0$}
		\label{fig:u0}
	\end{subfigure}
	\begin{subfigure}{0.49\textwidth}
		\centering
		\includegraphics[width=0.8\linewidth]{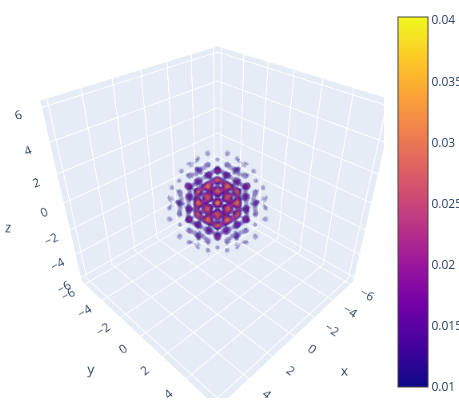}
		\caption{$t=0.4$}
		\label{fig:u1}
	\end{subfigure}
	\\
	\begin{subfigure}{0.49\textwidth}
		\centering
		\includegraphics[width=0.8\linewidth]{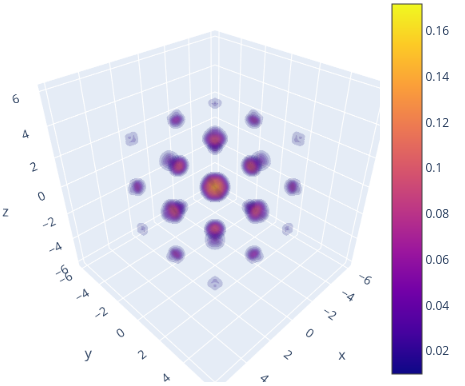}
		\caption{$t=0.8$}
		\label{fig:u2}
	\end{subfigure}
	\begin{subfigure}{0.49\textwidth}
		\centering
		\includegraphics[width=0.8\linewidth]{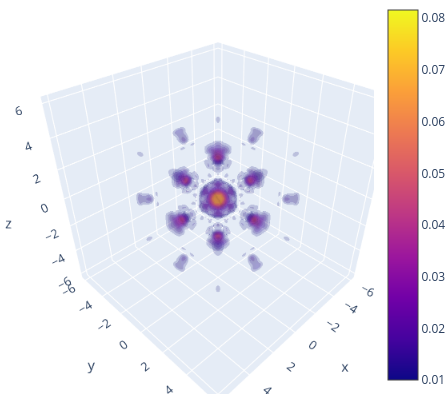}
		\caption{$t=1$}
		\label{fig:u4}
	\end{subfigure}
\caption{Snapshots of density of solution $|u|^2$ at different times.  Energy $\tilde E(u_\OD(t) ) = 39.848597$ .}
\end{figure}

The computation of the SLOD basis (one for the interior and all with support near the boundary) took about 1h, and all precomputations were done in about $1.7h$. The minimiser was then computed in less than $1h$ and the time-dependent problem solved in about $11h$, averaging about $300s$ per time step.

\appendix
\section{Notes on implementation}
Some novel key ideas for the efficient and reliable implementation are outlined here, in addition we recall that our code is made freely available at: \href{https://github.com/JWAER/SLOD\_BEC}{https://github.com/JWAER/SLOD\_BEC}.

\subsection{Specific mesh}
To facilitate the implementation and speed up certain computations, a Cartesian grid with nodes~$\mathcal{N}_H$ is used to define a simplicial subdivision~$\mathcal{T}_H$ of~$\D$. Since there are multiple such meshes, our specific choice in 2$d$ is illustrated in Fig.~\ref{fig:mesh_$2$d} where the blue dots represent coarse degrees of freedom and the red dots represent $\mathbb{P}^3_h(\mathcal{T}_h)$ degrees of freedom with $h=H$, the latter being used to represent the SLOD-basis functions. A similar figure but in 3$d$ is given in Fig.~\ref{fig:mesh_$3$d}.
\begin{figure}[H]
		\begin{subfigure}{0.49\textwidth}
		\centering
		\includegraphics[width=0.6\linewidth]{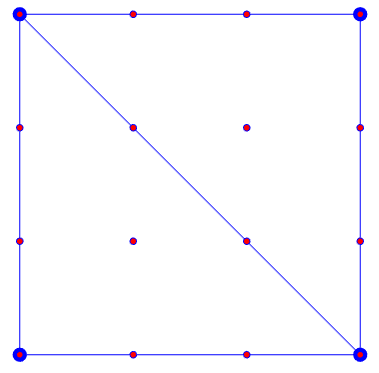}\vspace{5pt}
		\caption{2$d$ reference square with side length $H$. }
		\label{fig:mesh_$2$d}
	\end{subfigure}
	\begin{subfigure}{0.49\textwidth}
		\centering
		\includegraphics[width=0.7\linewidth]{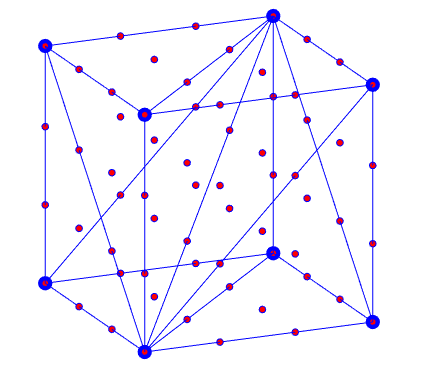}
		\caption{3$d$ reference cube with side length $H$. }
		\label{fig:mesh_$3$d}
	\end{subfigure}
\caption{Simplicial mesh of square and cube of side length $H$. Thick blue dots represent coarse degrees of freedom and thin red dots represent those of the $\mathbb{P}^3_h$-space used to represent the SLOD-basis functions with $h=H$.}
\end{figure}

\subsection{Assembly of nonlinear terms}\label{assembly_nonlinear}

As noted, a substantial increase in computational efficiency at low loss of accuracy can be achieved by replacing instances of $|u|^2$ by $P_\OD(|u|^2) $ and also precomputing the tensor,
\begin{align*}
	\omega^H_{\OD,ijk} = \langle \varphi_{\OD,i}\cdot \varphi_{\OD,j},\varphi_{\OD,k}\rangle .
\end{align*}
However, precomputing $\omega^H_\OD$ can become very expensive. Moreover, the approach followed in~\cite{HeWaDo22} to compute the elements of $\omega^H_\OD$ assumes high regularity of the OD-basis functions as higher order quadrature is used on the mesh $\mathcal{T}_H$. As an alternative we suggest instead computing 
\begin{align*}
	\omega^h_{ijk} = \langle v_{h,i}\cdot v_{h,j},v_{h,k}\rangle,
\end{align*} 
where $v_{h,i}$ denotes the $\mathbb{P}^3_h$-basis on the underlying mesh. Computing $\omega^h$ instead of $\omega^H_\OD$ is both simpler and faster as it suffices to compute a local tensor on a reference simplex. The assembly, e.g. of $\langle P_\OD(|u|^2)u,\varphi_{\OD,i}\rangle$, then becomes:
\begin{align*}
	\langle P_\OD(|u_\OD|^2)u_\OD,\varphi_{\OD,i}\rangle = \Phi_\OD \sum_{j,k} \omega^h_{ijk} (\Phi_\OD^T U)_j (\Phi_\OD^T U)_j,
\end{align*}
where the representation of the SLOD-basis functions on
the fine mesh is given by the matrix~$\Phi_\OD$, i.e., $\varphi_{\OD,i}(x) = \sum_j v_{h,j}(x)\Phi_{\OD,ij}$.

\subsection{Solvers}
When solving the time-dependent NLSE a nonlinear system of equations need to be solved in every time-step. In the previous work \cite{Superconv,HeWaDo22} a fixed point iteration is proposed involving a linear solve by means of a precomputed LU-factorization.  We use the same approach in $2$d, but found the approach to be infeasible in $3$d due to the memory requirement. More specifically, the $3$d example here presented would have involved 3 LU-factorization, each requiring about 24 GB of RAM memory. Therefore, in the case of $3$d, we instead use the iterative solver \textit{Induced Dimension Reduction method}, from the package \textit{IterativeSolvers.jl}. In each iteration of the minimisation algorithm the conjugate gradient method is used in both $2$d and $3$d.  

\section{Proof of Lemma~\ref{H^2-tilde}}\label{app:H^2-tilde}

For the proof of Lemma~\ref{H^2-tilde}, which is an $H$-independent bound of $\|\tilde u^0_\OD\|_{H^2}$, we follow Lemma~\ref{lem:H2_bound_u0_OD} to show that $-	\triangle \tilde{u}^0_\OD + V \tilde{u}^0_\OD$ is bounded in~$L^2$. Following the steps then gives us
\begin{equation}\label{eq:H^2-tilde_help}
	\| - 	\triangle \tilde u^0_\OD + V \tilde u^0_\OD\| \lesssim |\tilde \lambda^0_\OD| + 2 \beta \|P_\OD(|\tilde u^0_\OD|^2) \tilde u^0_\OD\|.
\end{equation} 
Here, the eigenvalue $\tilde \lambda^0_\OD$ is defined as in Lemma~\ref{lem:est_EW} and can be estimate by $|\tilde \lambda^0_\OD| \leq 5 \tilde E (\tilde u_\OD^0) \leq 5 E (u_\OD^0)$ with the order of the energies~\eqref{eq:order_energies}. Note that, the right-hand side can be bounded independently of~$H$. In the remainder of the proof, we will bound the second term $2 \beta \|P_\OD(|\tilde u^0_\OD|^2) \tilde u^0_\OD\|$. For $d=1$, in particular, we easily get 
\begin{equation*}
	\|P_\OD(|\tilde u^0_\OD|^2) \tilde u^0_\OD\| \lesssim \|\tilde u^0_\OD\|_{L^4}^2 \|\tilde u^0_\OD\|_{L^\infty} \lesssim \|\tilde u^0_\OD\|_{H^1}^3 \lesssim\tilde E(\tilde u^0_\OD) \leq E(u^0_\OD).
\end{equation*}
Hence, we will only consider the cases $d=2,3$ in the following. 

As the first step to estimate the second term of~\eqref{eq:H^2-tilde_help}, we shall derive an estimate of $\|\tilde{u}^0_\OD-u^0_\OD\|_{H^1}$, which is independent of the $H^2$-norm of $\tilde{u}^0_\OD$. For this, we use the two Gagliardo--Nirenberg-type interpolation inequalities
%
\begin{subequations}
	\begin{gather}
		\|\nabla v\|_{L^4} \lesssim  \|v\|_{H^2}^{d/4}\|v\|_{H^1}^{1-d/4},\\
		\|v\|_{L^\infty} \lesssim \|v\|_{W^{1,\theta}} \lesssim  \|v\|_{H^2}^{2/3}\|v\|_{H^1}^{1/3}\label{eq:Gag_Nir_Linfty}
	\end{gather}
\end{subequations}
with $\theta = \frac{6d}{3d-4}>d$,
see \cite[Th.~4.12 \&~5.8]{AdaF03}. 
The two estimates then imply
\begin{align*}
	\|u\|_{L^\infty}\|u\|_{H^2} + \|\nabla u\|_{L^4}^2 + \|V\|\|u\|_{L^\infty}^2
	\lesssim\, & \|u\|_{H^2}^{5/3}\|u\|_{H^1}^{1/3} + 
	\|u\|_{H^2}^{d/2}\|u\|_{H^1}^{2-d/2} + 
	\|u\|_{H^2}^{4/3}\|u\|_{H^1}^{2/3}\\
	\lesssim\, & \|u\|_{H^2}^{5/3}\|u\|_{H^1}^{1/3}.
\end{align*}
In particular, by the lines of Lemma~\ref{lem:density_replacement}, estimate~\eqref{eqn:est_u2} can be refined to 
\begin{equation*}
	\|\,|u|^2 - P_\OD(|u|^2)\|_{L^2} \lesssim H^2\|u\|_{H^2}^{5/3}\|u\|_{H^1}^{1/3}
\end{equation*}
In the special case of $u=v_\OD \in \VOD$, we have the inverse estimate 
\begin{equation}\label{eq:inverse_H2}
	\|v_\OD \|_{H^2}\lesssim H^{-1}\|v_\OD \|_{H^1},
\end{equation}
see \cite[p.~545 f.]{Superconv}. Hence, we get even an $H^2(\D)$-independent estimate by
\begin{equation*}
	\|\,|v_\OD|^2 - P_\OD(|v_\OD|^2)\| \lesssim H^{1/3}\|v_\OD\|_{H^1}^{2}.
\end{equation*}
By the steps of the proof of Lemma~\ref{lem:crude_est}, we then have 
\begin{align}\label{eqn:est_u_tildeu_new}
	\|\tilde{u}^0_\OD-u^0_\OD\|_{H^1} \lesssim H^{1/3}\|\tilde{u}^0_\OD\|_{H^1}^{2} + H^3
\end{align}
with a constant independent of $H$ and $\tilde{u}^0_\OD$.


For the second term $\|P_\OD(|\tilde u^0_\OD|^2) \tilde u^0_\OD\|$ of~\eqref{eq:H^2-tilde_help}, we now get
\begin{align*}
	\|P_\OD(|\tilde u^0_\OD|^2) \tilde u^0_\OD\| \leq\, & \|(P_\OD(|\tilde u^0_\OD|^2)-|\tilde u^0_\OD|^2)\tilde u^0_\OD\| + \|\tilde u^0_\OD\|_{L^6}^3\\
	\lesssim\, & \|P_\OD(|\tilde u^0_\OD|^2)-|\tilde u^0_\OD|^2\|\,\|\tilde u^0_\OD\|_{L^\infty} + \|\tilde u^0_\OD\|_{H^1}^3\\
	\lesssim\, & H^{1/3}\|\tilde{u}_\OD^0\|_{H^1}^2\|\tilde u^0_\OD\|_{L^\infty} + \|\tilde u^0_\OD\|_{H^1}^3\\
	\lesssim\, & H^{1/3}\|\tilde{u}_\OD^0\|_{H^1}^2	\big(\|\tilde u^0_\OD-u^0_\OD\|_{L^\infty} + \|u_\OD^0\|_{L^\infty}\big) +  \|\tilde u^0_\OD\|_{H^1}^3
\end{align*}
Here, every term besides $H^{1/3}\|\tilde u^0_\OD-u^0_\OD\|_{L^\infty}$ can be bounded independent of $H$ by previous results, since $\|\tilde{u}^0_\OD\|_{H^1}^2 \lesssim E(u^0_\OD)$ and $\|u^0_\OD\|_{L^\infty} \lesssim \|u^0_\OD\|_{H^2}$ hold. Finally, for the difference in $L^\infty$, we use the estimates~\eqref{eq:Gag_Nir_Linfty},  \eqref{eq:inverse_H2}, and~\eqref{eqn:est_u_tildeu_new} to derive 
\begin{align*}
	H^{1/3}\|\tilde u^0_\OD-u^0_\OD\|_{L^\infty} \lesssim\, & H^{1/3} \|\tilde u^0_\OD-u^0_\OD\|_{H^2}^{2/3}\|\tilde u^0_\OD-u^0_\OD\|_{H^1}^{1/3}\\
	\lesssim\, & H^{-1/3} \|\tilde u^0_\OD-u^0_\OD\|_{H^1}\\
	\lesssim\, & \|\tilde{u}^0_\OD\|_{H^1}^{2} + H^{8/3}.
\end{align*}
Again, $\|\tilde{u}^0_\OD\|_{H^1}^{2}$ is bounded by a multiple of $E(u^0_\OD)$. In summary, $	\|P_\OD(|\tilde u^0_\OD|^2) \tilde u^0_\OD\|$ and therefore $\|-	\triangle \tilde u^0_\OD + V \tilde u^0_\OD\|$ can be bounded independent of $H$. 
The bound for $\|\tilde u^0_\OD\|_{H^2}$ follows then by the steps of Lemma~\ref{lem:H2_bound_u0_OD}.\hfill \qedsymbol
\end{document}

%% file: P1_OD_single.tex
%
%
\definecolor{mycolor1}{RGB}{8,118,188}%
\definecolor{mycolor2}{RGB}{234,174,40}%
\begin{tikzpicture}
\begin{axis}[%
	width=2.75in,
	height=1.7in,
	at={(0in,0in)},
	scale only axis,
	xmin=0,
	xmax=1,
	ymin=-0.55,
	ymax=1.05,
	axis background/.style={fill=white},
	title style={font=\bfseries},
	xticklabels={,,},
	yticklabels={,,},
	axis lines=none, 
	xtick=\empty, 
	ytick=\empty
	]
\addplot [color=gray, forget plot]
table[row sep=crcr]{%
	0 0.035\\
	0 -0.035\\
	0 0\\
	0.1 0\\
	0.1 0.035\\
	0.1 -0.035\\
	0.1 0\\
	0.2 0\\
	0.2 0.035\\
	0.2 -0.035\\
	0.2 0\\
	0.3 0\\
	0.3 0.035\\
	0.3 -0.035\\
	0.3 0\\
	0.4 0\\
	0.4 0.035\\
	0.4 -0.035\\
	0.4 0\\	
	0.5 0\\
	0.5 0.035\\
	0.5 -0.035\\
	0.5 0\\
	0.6 0\\
	0.6 0.035\\
	0.6 -0.035\\
	0.6 0\\
	0.7 0\\
	0.7 0.035\\
	0.7 -0.035\\
	0.7 0\\
	0.8 0\\
	0.8 0.035\\
	0.8 -0.035\\
	0.8 0\\
	0.9 0\\
	0.9 0.035\\
	0.9 -0.035\\
	0.9 0\\
	1 0\\
	1 0.035\\
	1 -0.035\\
};

\addplot [color=mycolor2, forget plot, dashed, line width= 1.5]
table[row sep=crcr]{%
0 0\\
0.4 0\\
0.5 0.98\\
0.6 0\\
1 0\\	
};

\addplot [color=mycolor1, forget plot, line width= 1.5]
  table[row sep=crcr]{%
0	0\\
0.0050251256281407	0.0107695710299908\\
0.0100502512562814	0.0215391420599816\\
0.0150753768844221	0.0323087130899724\\
0.0201005025125628	0.0430782841199633\\
0.0251256281407035	0.0538478551499541\\
0.0301507537688442	0.0646174261799449\\
0.0351758793969849	0.0753869972099357\\
0.0402010050251256	0.0861565682399265\\
0.0452261306532663	0.0969261392699173\\
0.050251256281407	0.107695710299908\\
0.0552763819095477	0.118465281329899\\
0.0603015075376884	0.12923485235989\\
0.0653266331658292	0.140004423389881\\
0.0703517587939698	0.150773994419871\\
0.0753768844221105	0.161543565449862\\
0.0804020100502513	0.172313136479853\\
0.085427135678392	0.183082707509844\\
0.0904522613065327	0.193852278539835\\
0.0954773869346734	0.204621849569825\\
0.100502512562814	0.215391420599816\\
0.105527638190955	0.226160991629807\\
0.110552763819095	0.236930562659798\\
0.115577889447236	0.247700133689789\\
0.120603015075377	0.25846970471978\\
0.125628140703518	0.26923927574977\\
0.130653266331658	0.280008846779761\\
0.135678391959799	0.290778417809752\\
0.14070351758794	0.301547988839743\\
0.14572864321608	0.312317559869734\\
0.150753768844221	0.323087130899724\\
0.155778894472362	0.333856701929715\\
0.160804020100503	0.344626272959706\\
0.165829145728643	0.355395843989697\\
0.170854271356784	0.366165415019688\\
0.175879396984925	0.376934986049679\\
0.180904522613065	0.387704557079669\\
0.185929648241206	0.39847412810966\\
0.190954773869347	0.409243699139651\\
0.195979899497487	0.420013270169642\\
0.201005025125628	0.430782841199633\\
0.206030150753769	0.441552412229623\\
0.21105527638191	0.452321983259614\\
0.21608040201005	0.463091554289605\\
0.221105527638191	0.473861125319596\\
0.226130653266332	0.484630696349586\\
0.231155778894472	0.495400267379577\\
0.236180904522613	0.506169838409568\\
0.241206030150754	0.516939409439559\\
0.246231155778894	0.52770898046955\\
0.251256281407035	0.53847855149954\\
0.256281407035176	0.549248122529531\\
0.261306532663317	0.560017693559522\\
0.266331658291457	0.570787264589513\\
0.271356783919598	0.581556835619504\\
0.276381909547739	0.592326406649494\\
0.281407035175879	0.603095977679485\\
0.28643216080402	0.613865548709476\\
0.291457286432161	0.624635119739467\\
0.296482412060302	0.635404690769458\\
0.301507537688442	0.646174261799449\\
0.306532663316583	0.65694383282944\\
0.311557788944724	0.667713403859431\\
0.316582914572864	0.678482974889421\\
0.321608040201005	0.689252545919412\\
0.326633165829146	0.700022116949403\\
0.331658291457286	0.710791687979394\\
0.336683417085427	0.721561259009385\\
0.341708542713568	0.732330830039375\\
0.346733668341709	0.743100401069366\\
0.351758793969849	0.753869972099357\\
0.35678391959799	0.764639543129348\\
0.361809045226131	0.775409114159339\\
0.366834170854271	0.786178685189329\\
0.371859296482412	0.79694825621932\\
0.376884422110553	0.807717827249311\\
0.381909547738693	0.818487398279302\\
0.386934673366834	0.829256969309293\\
0.391959798994975	0.840026540339283\\
0.396984924623116	0.850796111369274\\
0.402010050251256	0.861565102235005\\
0.407035175879397	0.872310378886604\\
0.412060301507538	0.882979508979066\\
0.417085427135678	0.893518102113007\\
0.422110552763819	0.903871767889043\\
0.42713567839196	0.913986115907791\\
0.4321608040201	0.923806755769865\\
0.437185929648241	0.933279297075884\\
0.442211055276382	0.942349349426462\\
0.447236180904523	0.950962522422216\\
0.452261306532663	0.959064425663762\\
0.457286432160804	0.966600668751717\\
0.462311557788945	0.973516861286696\\
0.467336683417085	0.979758612869316\\
0.472361809045226	0.985271533100193\\
0.477386934673367	0.990001231579944\\
0.482412060301508	0.993893317909183\\
0.487437185929648	0.996893401688528\\
0.492462311557789	0.998947092518595\\
0.49748743718593	1\\
0.50251256281407	1\\
0.507537688442211	0.998947092518595\\
0.512562814070352	0.996893401688528\\
0.517587939698492	0.993893317909182\\
0.522613065326633	0.990001231579943\\
0.527638190954774	0.985271533100192\\
0.532663316582915	0.979758612869315\\
0.537688442211055	0.973516861286695\\
0.542713567839196	0.966600668751715\\
0.547738693467337	0.959064425663761\\
0.552763819095477	0.950962522422215\\
0.557788944723618	0.94234934942646\\
0.562814070351759	0.933279297075882\\
0.5678391959799	0.923806755769864\\
0.57286432160804	0.91398611590779\\
0.577889447236181	0.903871767889043\\
0.582914572864322	0.893518102113007\\
0.587939698492462	0.882979508979066\\
0.592964824120603	0.872310378886604\\
0.597989949748744	0.861565102235005\\
0.603015075376884	0.850796111369275\\
0.608040201005025	0.840026540339284\\
0.613065326633166	0.829256969309293\\
0.618090452261307	0.818487398279302\\
0.623115577889447	0.807717827249311\\
0.628140703517588	0.79694825621932\\
0.633165829145729	0.78617868518933\\
0.638190954773869	0.775409114159339\\
0.64321608040201	0.764639543129348\\
0.648241206030151	0.753869972099357\\
0.653266331658292	0.743100401069366\\
0.658291457286432	0.732330830039376\\
0.663316582914573	0.721561259009385\\
0.668341708542714	0.710791687979394\\
0.673366834170854	0.700022116949403\\
0.678391959798995	0.689252545919412\\
0.683417085427136	0.678482974889421\\
0.688442211055276	0.66771340385943\\
0.693467336683417	0.65694383282944\\
0.698492462311558	0.646174261799449\\
0.703517587939699	0.635404690769458\\
0.708542713567839	0.624635119739467\\
0.71356783919598	0.613865548709477\\
0.718592964824121	0.603095977679486\\
0.723618090452261	0.592326406649495\\
0.728643216080402	0.581556835619504\\
0.733668341708543	0.570787264589513\\
0.738693467336683	0.560017693559522\\
0.743718592964824	0.549248122529531\\
0.748743718592965	0.538478551499541\\
0.753768844221106	0.52770898046955\\
0.758793969849246	0.516939409439559\\
0.763819095477387	0.506169838409568\\
0.768844221105528	0.495400267379578\\
0.773869346733668	0.484630696349587\\
0.778894472361809	0.473861125319596\\
0.78391959798995	0.463091554289605\\
0.78894472361809	0.452321983259614\\
0.793969849246231	0.441552412229623\\
0.798994974874372	0.430782841199633\\
0.804020100502513	0.420013270169641\\
0.809045226130653	0.409243699139651\\
0.814070351758794	0.39847412810966\\
0.819095477386935	0.387704557079669\\
0.824120603015075	0.376934986049678\\
0.829145728643216	0.366165415019687\\
0.834170854271357	0.355395843989697\\
0.839195979899497	0.344626272959706\\
0.844221105527638	0.333856701929715\\
0.849246231155779	0.323087130899724\\
0.85427135678392	0.312317559869734\\
0.85929648241206	0.301547988839743\\
0.864321608040201	0.290778417809752\\
0.869346733668342	0.280008846779761\\
0.874371859296482	0.26923927574977\\
0.879396984924623	0.258469704719779\\
0.884422110552764	0.247700133689789\\
0.889447236180904	0.236930562659798\\
0.894472361809045	0.226160991629807\\
0.899497487437186	0.215391420599816\\
0.904522613065327	0.204621849569825\\
0.909547738693467	0.193852278539834\\
0.914572864321608	0.183082707509844\\
0.919597989949749	0.172313136479853\\
0.924623115577889	0.161543565449862\\
0.92964824120603	0.150773994419871\\
0.934673366834171	0.14000442338988\\
0.939698492462312	0.12923485235989\\
0.944723618090452	0.118465281329899\\
0.949748743718593	0.107695710299908\\
0.954773869346734	0.0969261392699173\\
0.959798994974874	0.0861565682399264\\
0.964824120603015	0.0753869972099358\\
0.969849246231156	0.0646174261799449\\
0.974874371859296	0.0538478551499541\\
0.979899497487437	0.0430782841199632\\
0.984924623115578	0.0323087130899724\\
0.989949748743719	0.0215391420599817\\
0.994974874371859	0.0107695710299909\\
1	0\\
};
\end{axis}
\end{tikzpicture}%

%% file: P1_SLOD_single.tex
%
%
\definecolor{mycolor1}{RGB}{8,118,188}%
\definecolor{mycolor2}{RGB}{234,174,40}%
\begin{tikzpicture}

\begin{axis}[%
	width=2.75in,
	height=1.7in,
	at={(0in,0in)},
	scale only axis,
	xmin=0,
	xmax=1,
	ymin=-0.55,
	ymax=1.05,
	axis background/.style={fill=white},
	title style={font=\bfseries},
	xticklabels={,,},
	yticklabels={,,},
	axis lines=none, 
	xtick=\empty, 
	ytick=\empty
	]
\addplot [color=gray, forget plot]
table[row sep=crcr]{%
	0 0.035\\
	0 -0.035\\
	0 0\\
	0.1 0\\
	0.1 0.035\\
	0.1 -0.035\\
	0.1 0\\
	0.2 0\\
	0.2 0.035\\
	0.2 -0.035\\
	0.2 0\\
	0.3 0\\
	0.3 0.035\\
	0.3 -0.035\\
	0.3 0\\
	0.4 0\\
	0.4 0.035\\
	0.4 -0.035\\
	0.4 0\\	
	0.5 0\\
	0.5 0.035\\
	0.5 -0.035\\
	0.5 0\\
	0.6 0\\
	0.6 0.035\\
	0.6 -0.035\\
	0.6 0\\
	0.7 0\\
	0.7 0.035\\
	0.7 -0.035\\
	0.7 0\\
	0.8 0\\
	0.8 0.035\\
	0.8 -0.035\\
	0.8 0\\
	0.9 0\\
	0.9 0.035\\
	0.9 -0.035\\
	0.9 0\\
	1 0\\
	1 0.035\\
	1 -0.035\\
};

\addplot [color=mycolor2, forget plot, dashed, line width= 1.5]
table[row sep=crcr]{%
	0 0\\
	0.3 0\\
	0.4 -0.49\\
	0.5 0.975\\
	0.6 -0.49\\
	0.7 0\\
	1 0\\	
};

\addplot [color=mycolor1, forget plot, line width= 1.5]
  table[row sep=crcr]{%
0	0\\
0.3	0\\
0.311557788944724	0.000386340790682751\\
0.316582914572864	0.00114111358549589\\
0.321608040201005	0.00252459909958966\\
0.326633165829146	0.00472731633880667\\
0.331658291457286	0.00793978430898949\\
0.336683417085427	0.0123525220159808\\
0.341708542713568	0.0181560484656232\\
0.346733668341709	0.0255408826637593\\
0.351758793969849	0.0346975436162317\\
0.35678391959799	0.045816550328883\\
0.361809045226131	0.0590884218075559\\
0.366834170854271	0.0747036770580929\\
0.371859296482412	0.0928528350863368\\
0.376884422110553	0.11372641489813\\
0.381909547738693	0.137514935499315\\
0.386934673366834	0.164408915895735\\
0.391959798994975	0.194598875093232\\
0.396984924623116	0.228275332097649\\
0.402010050251256	0.265620677103912\\
0.407035175879397	0.306501292782591\\
0.412060301507538	0.350373056853\\
0.417085427135678	0.396664412297608\\
0.422110552763819	0.444803802098891\\
0.42713567839196	0.49421966923932\\
0.4321608040201	0.544340456701365\\
0.437185929648241	0.594594607467501\\
0.442211055276382	0.644410564520199\\
0.447236180904523	0.69321677084193\\
0.452261306532663	0.740441669415169\\
0.457286432160804	0.785513703222385\\
0.462311557788945	0.827861315246052\\
0.467336683417085	0.866912948468642\\
0.472361809045226	0.902097045872627\\
0.477386934673367	0.932842050440479\\
0.482412060301508	0.95857640515467\\
0.487437185929648	0.978728552997672\\
0.492462311557789	0.992726936951958\\
0.49748743718593	1\\
0.50251256281407	1\\
0.507537688442211	0.992726936951958\\
0.512562814070352	0.978728552997672\\
0.517587939698492	0.958576405154669\\
0.522613065326633	0.932842050440478\\
0.527638190954774	0.902097045872627\\
0.532663316582915	0.866912948468642\\
0.537688442211055	0.827861315246052\\
0.542713567839196	0.785513703222384\\
0.547738693467337	0.740441669415168\\
0.552763819095477	0.693216770841931\\
0.557788944723618	0.644410564520199\\
0.562814070351759	0.594594607467501\\
0.5678391959799	0.544340456701366\\
0.57286432160804	0.49421966923932\\
0.577889447236181	0.444803802098893\\
0.582914572864322	0.39666441229761\\
0.587939698492462	0.350373056853001\\
0.592964824120603	0.306501292782593\\
0.597989949748744	0.265620677103914\\
0.603015075376884	0.228275332097651\\
0.608040201005025	0.194598875093234\\
0.613065326633166	0.164408915895737\\
0.618090452261307	0.137514935499317\\
0.623115577889447	0.113726414898131\\
0.628140703517588	0.0928528350863378\\
0.633165829145729	0.0747036770580943\\
0.638190954773869	0.059088421807557\\
0.64321608040201	0.045816550328884\\
0.648241206030151	0.0346975436162325\\
0.653266331658292	0.0255408826637599\\
0.658291457286432	0.0181560484656238\\
0.663316582914573	0.0123525220159813\\
0.668341708542714	0.00793978430898985\\
0.673366834170854	0.0047273163388069\\
0.678391959798995	0.00252459909958983\\
0.683417085427136	0.00114111358549602\\
0.688442211055276	0.000386340790682812\\
0.693467336683417	6.97617093076326e-05\\
0.698492462311558	8.57335527860722e-07\\
0.7 0\\
1	0\\
};
\end{axis}
\end{tikzpicture}%